\documentclass[11pt]{amsart}
\usepackage{latexsym}
\usepackage{amsmath}
\usepackage{amssymb}
\usepackage{color}

\newtheorem{Theorem}{Theorem}[section]
\newtheorem{lemma}[Theorem]{Lemma}
\newtheorem{prop}[Theorem]{Proposition}

\numberwithin{equation}{section}

\def\C{{\mathbb C}}
\def\R{{\mathbb R}}
\def\Re{{\rm Re}\:}
\def\Im{{\rm Im}\:}
\def\sgn{{\rm sgn}\:}
\def\cal{\mathcal}

\newcommand*{\defeq}{\mathrel{\vcenter{\baselineskip0.5ex \lineskiplimit0pt

                     \hbox{\scriptsize.}\hbox{\scriptsize.}}}%
                     =}
                     
\newcommand*{\qefed}{=\mathrel{\vcenter{\baselineskip0.5ex \lineskiplimit0pt

                     \hbox{\scriptsize.}\hbox{\scriptsize.}}}%
                     }

\begin{document}

\title[Resolvent estimates]{Resolvent estimates for the Schr\"odinger operator 
with $L^\infty$ electric and magnetic potentials and applications to the local energy decay}

\author[A. Larra\'in-Hubach, J. Shapiro, G. Vodev]{Andr\'es Larra\'in-Hubach, Jacob Shapiro and Georgi Vodev}

\address {Department of Mathematics, University of Dayton, Dayton, OH 45469-2316, USA}
\email{alarrainhubach1@udayton.edu}

\address {Department of Mathematics, University of Dayton, Dayton, OH 45469-2316, USA}
\email{jshapiro1@udayton.edu}

\address {Universit\'e de Nantes, Laboratoire de Math\'ematiques Jean Leray, 2 rue de la Houssini\`ere, BP 92208, 44322 Nantes Cedex 03, France}
\email{Georgi.Vodev@univ-nantes.fr}

\date{}
\begin{abstract} We extend the results in \cite{kn:LSV} to a larger class of electric potentials
$V\in L^\infty(\R^d;\R)$, $d\ge 3$, and magnetic potentials $b\in L^\infty(\R^d;\R^d)$ such that 
$V(x), b(x)=O_k\left(|x|^{-k}\right)$, $|x|\gg 1$, for every integer $k$. More precisely, we prove estimates for the derivatives of the 
weighted resolvent of the 
corresponding magnetic Schr\"odinger operator, which are uniform with respect to both the spectral parameter and
the order of derivation. We also show that these resolvent estimates still hold for the
Dirichlet self-adjoint realization of the Schr\"odinger operator in the exterior of a non-trapping obstacle in $\R^d$, $d\ge 2$,
provided the magnetic potential is supposed identically zero. As an application of these resolvent estimates, 
we obtain the rate of decay of the local energy of solutions to the corresponding wave equation.
In particular, we show that for potentials satisfying $|V(x)|+|b(x)|\le Ce^{-c|x|^s}$, $c,C>0$, $0<s<1$,
the rate of decay of the local energy is $e^{-c_0t^s}$ with some constant $c_0>0$, where $t\gg 1$ is the time variable.


\end{abstract}

\maketitle

\setcounter{section}{0}

\section{Introduction} \label{introduction section}
Our goal in this paper is to extend the recent results in \cite{kn:LSV} to a larger class of electric and magnetic potentials. 
Let $\cal{O}\subseteq\mathbb{R}^d$, $d\ge 2$, be a possibly empty, bounded domain with smooth boundary such that 
$\Omega=\mathbb{R}^d\setminus\cal{O}$ is connected. In \cite{kn:LSV}, the magnetic Schr\"odinger operator
\begin{equation} \label{eq:1.1}
P=(i\nabla+b(x))^2+V(x) : L^2(\Omega) \to L^2(\Omega),
\end{equation}
is studied under various conditions on the obstacle $\cal{O}$, the magnetic potential $b\in L^\infty(\Omega; \mathbb{R}^d)$, 
and the electric potential $V\in L^\infty(\Omega;\mathbb{R})$. It is supposed in \cite{kn:LSV}, and we continue to assume in this work that
 either

a) $\cal{O} = \emptyset$ (so that $\Omega = \R^d$), $d\ge 3$, and $b$ is not identically zero,

\noindent
or

b) $\cal{O}$ is not empty and $b\equiv 0$. 

In \cite{kn:LSV} the local energy decay for the solution to the wave equation
\begin{equation} \label{eq:1.2}
\begin{cases}
(\partial^2_t + P)u(t,x) = 0 & \text{in } \R \times \Omega, \\
u(t,x) = 0 & \text{on } \R \times \partial \Omega, \\
u(0,x) = f_1(x), \, \partial_tu(0,x) = f_2(x) & \text{in } \Omega,  
\end{cases}
\end{equation}
is studied for potentials satisfying
\begin{equation}\label{eq:1.3}
|V(x)|+|b(x)|\le Ce^{-c\langle x\rangle},
\end{equation}
where $\langle x \rangle \defeq (1 + |x|^2)^{1/2}$ and $C,c>0$ are some constants.
In what follows $\|\cdot\|$ and $\|\cdot\|_1$ denote the operator norms $L^2(\Omega)\to L^2(\Omega)$ and $H^1(\Omega)\to L^2(\Omega)$, respectively. 

Denote by $\widetilde{P}$ the self-adjoint Dirichlet realization of $-\Delta$ on the Hilbert space $L^2(\Omega)$. Let $\chi \in C^\infty(\mathbb{R}^d; [0,1])$ be compactly supported and equal to $1$ near $\overline{\mathcal{O}}$. Define the multiplication operator on $L^2(\Omega)$ by $u \mapsto \chi|_\Omega u$.

In case b), we assume $\mathcal{O}$ is \textit{non-trapping}. More precisely, the operator $\widetilde{P}$ fits into the framework of a \textit{black box Hamiltonian} in the sense of Sjöstrand and Zworski \cite{kn:SZ91}, as defined in \cite[Definition 4.1]{kn:dz}. Under this setup, the cutoff resolvent $\chi (\widetilde{P} - \lambda^2)^{-1} \chi : L^2(\Omega) \to D(\widetilde{P})$ admits a meromorphic continuation from the lower half-plane $\{\operatorname{Im} \lambda < 0\}$ to the entire complex plane $\mathbb{C}$ when $d$ is odd, and to the Riemann surface of the logarithm when $d$ is even, by the analytic Fredholm theorem \cite[Theorem 4.4]{kn:dz}. The poles of this continuation are referred to as its \textit{resonances}.

Our non-trapping assumption takes the form of the following high-frequency resolvent estimate: there exist constants $C > 0$ and $\lambda_0 > 0$ such that for all $\lambda \ge \lambda_0$,
\begin{equation}\label{eq:1.4}
\left\|\chi(\widetilde{P} - \lambda^2)^{-1} \chi\right\| \le C \lambda^{-1}.
\end{equation}
We note that (\ref{eq:1.4}) also holds for $0<\lambda\le\lambda_0$ for an arbitrary obstacle, see \cite[Appendix B]{kn:B1}.

In what follows $P$ denotes the self-adjoint realization of the operator $(i\nabla+b)^2+V$ on the Hilbert space
$L^2(\mathbb{R}^d)$ in the case a), and the Dirichlet self-adjoint realization of $-\Delta+V$ in the case b). For the case a), see \cite[Appendix A]{kn:LSV} for the construction of $P$ via a sesquilinear form. For the case b), we recall that the domain of $P$ is the intersection $H^1_0(\Omega) \cap H^2(\Omega)$ of Sobolev spaces. In both cases we suppose that $P\ge 0$, for which a sufficient condition is $V\ge 0$. 
Then the solution to the wave equation (\ref{eq:1.2}) can be expressed by the formula
\begin{equation}\label{eq:1.5}
u(t)=\cos\left(t\sqrt{P}\right)f_1+P^{-1/2}\sin\left(t\sqrt{P}\right)f_2.
\end{equation}
For short-range electric and magnetic potentials, it follows from the resolvent estimates proved in \cite{kn:LSV} 
 (see Theorem \ref{4.1}) and the limiting absorption principle (which we review in Appendix \ref{limiting absorption appendix}) that the limit
\begin{equation*}
(P-\lambda^2)^{-1}=\lim_{\varepsilon\to 0}(P-(\lambda-i\varepsilon)^2)^{-1}: \langle x\rangle^{-s}H^{-1}(\mathbb{R}^d)\to \langle x\rangle^{s}H^1(\mathbb{R}^d),\quad s>\frac{1}{2},
\end{equation*}
 exists for all $\lambda\in\R$, $\lambda\neq 0$. Our goal in this paper is to study the operator-valued function 
 $(P-\lambda^2)^{-1}$ for electric and magnetic potentials satisfying
\begin{equation}\label{eq:1.6}
|V(x)|+|b(x)|\le C\Theta(c\langle x\rangle),
\end{equation}
where $C,c>0$ are some constants and the function $\Theta\in C^2([0, \infty); (0, \infty))$ is decreasing such that 
$\Theta(r)=O_k((r+1)^{-k})$ for every integer $k\ge 0$. We suppose that the first and the second derivatives of $\Theta$ satisfy
\begin{equation}\label{eq:1.7}
\left|\partial_r^j\Theta(r)\right|\le\widetilde C\Theta(r),\quad\forall r\ge 0,\quad j=1,2,
\end{equation}
with some constant $\widetilde C>0$.
We also suppose that there are constants $C_1,C_2>0$ such that
the function $\Theta$ satisfies the inequality
\begin{equation}\label{eq:1.8}
\Theta(r_1)\Theta(r_2)\le C_1\Theta(C_2(r_1+r_2)),\quad\forall r_1,r_2>0.
\end{equation}
Set 
\begin{equation*}
\widetilde m_k=\sup_{r\ge 0}\,(r+1)^{k}\sqrt{\Theta(r)}.
\end{equation*}
Let $m=\{m_k\}_{k=0}^\infty$, $m_0=1$, be an increasing sequence satisfying the inequalities
\begin{equation}\label{eq:1.9}
k!+\widetilde m_k+m_{k+2}+\sqrt{m_{2k}}\le C_3^{k+1}m_k,
\end{equation}
\begin{equation}\label{eq:1.10}
\frac{m_\nu m_{k-\nu}}{\nu!(k-\nu)!}\le C_4\frac{m_k}{k!},\quad 0\le\nu\le k,
\end{equation}
for all integers $k\ge 0$ with constants $C_3,C_4>0$ independent of $k$.

Our main example of a function $\Theta$ satisfying the conditions 
(\ref{eq:1.7}), (\ref{eq:1.8}), (\ref{eq:1.9}) and (\ref{eq:1.10}) is $\Theta(r)=e^{-(r+1)^s}$ with $0<s\le 1$. Indeed, we have
\begin{equation*}
\widetilde m_k=\sup_{r\ge 0}\,(r+1)^{k}e^{-(r+1)^s/2}=2^{k/s}\sup_{\sigma>0}\,\sigma^{k/s}e^{-\sigma}
 = (k/s)^{k/s} e^{-k/s} 
\end{equation*}
 Stirling's formula says that $k! \sim \sqrt{2 \pi k } (k/e)^k$ as $k \to \infty$, so 
\begin{equation*}
(k/s)^{k/s} e^{-k/s} \sim s^{-k/s} (2\pi k)^{-1/2s} (k!)^{1/s}.
\end{equation*}
Therefore, in this case we can take $m_k=(k!)^{1/s}$. A wider class of admissible functions is $\Theta(r) = e^{-(r +1)^s(\log(r + e))^\beta}$ for $0 < s \le 1$ and $\beta > 0$, for which $m_k = (k!)^{1/s} (\log(k +e))^{-\beta k /s}$.

Set $\mu(x)=\sqrt{\Theta(c\langle x\rangle)}$. To define the operator-valued function 
 $\mu(P-\lambda^2)^{-1}\mu:L^2\to L^2$ at $\lambda=0$, we need to suppose that zero is neither 
an eigenvalue nor a resonance of $P$.
More precisely, we need to suppose that 
there are constants $C>0$ and $0<\delta_0\ll 1$ so that 
the low-frequency resolvent bound 
\begin{equation}\label{eq:1.11}
\sum_{\ell=0}^1\left\|\mu\nabla^\ell(P-\lambda^2\pm i\varepsilon)^{-1}\mu\right\|\le C,\quad 0<\lambda\le\delta_0,
\end{equation}
holds uniformly in $0<\varepsilon\le 1$. The estimate \eqref{eq:1.11} is established in \cite[Section 5]{kn:LSV} for $d\ge 5$, provided $V$ is nonnegative and $|V(x)| + |b(x)| \le \langle x \rangle^{-\rho}$ with $\rho > \max(3, d/2)$. Our main result is the following

\begin{Theorem} \label{1.1}
Assume the conditions \eqref{eq:1.4}, \eqref{eq:1.6}, \eqref{eq:1.7}, \eqref{eq:1.8}, (\ref{eq:1.9}) and \eqref{eq:1.10} are fulfilled. 
Let $\ell\in\{0,1\}$. Then, given any $\delta>0$ there is a constant $C=C_\delta>0$ so that the estimate 
\begin{equation}\label{eq:1.12}
\left\|\frac{d^k}{d\lambda^k}\mu\nabla^{\ell}(P-\lambda^2)^{-1}\mu\right\|\le C^{k+1}m_k(|\lambda|+1)^{\ell-1}
\end{equation}
holds for all integers $k\ge 0$ and all $\lambda\in\R$, $|\lambda|\ge\delta$. 
 If the dimension $d$ is odd and the condition
 (\ref{eq:1.11}) is assumed, then the estimate (\ref{eq:1.12}) holds for all $\lambda\in\R$.
 \end{Theorem}

Let $\mathcal{C}_m^\infty(\R)$ be the set of all functions $a\in C^\infty(\R)$ satisfying
\begin{equation*}
\left|\partial_\lambda^ka(\lambda)\right|\le C^{k+1}m_k,\quad\forall\lambda\in\R,\,\,\forall k\in\mathbb{N},
\end{equation*}
 with a constant $C>0$ independent of $k$ and $\lambda$.
 If $\Theta(r)=e^{-(r+1)^s}$ with $0<s\le 1$, we have
$m_k=(k!)^{1/s}$, so in this case $\mathcal{C}_m^\infty(\R)$ is a subspace of the 
Gevrey functions $G^s(\R)$ if $s<1$ and of the analytic functions if $s=1$. 

As a consequence of Theorem \ref{1.1} we obtain the rate of the decay
of the local energy of the solution to the wave equation (\ref{eq:1.2}). 
We suppose that there exists a function $\zeta\in \mathcal{C}_m^\infty(\R)$, $0\le\zeta\le 1$, $\int_{\R}\zeta(\sigma)d\sigma=1$,
$\zeta(\sigma)=0$ near $(-\infty, 1/2] \cup [2, \infty)$. When $\Theta(r)=e^{-(r+1)^s}$, such a function exists if $0<s<1$ (see Appendix \ref{gevrey cutoff appendix} for a brief construction), but it does not if $s=1$. Set
\begin{equation*}
\psi(\lambda)=\int_{-\infty}^\lambda\zeta(\sigma)d\sigma.
\end{equation*}
Given any $t>1$, we let $k(t)$ be the biggest
integer $k$ such that $m_k^{1/k}\le t$. 
When $\Theta(r)=e^{-(r+1)^s}$ with $0<s\le 1$, it is easy to see that $k(t)\sim c_1t^s$ with some constant $c_1>0$.
We have the following 

\begin{Theorem} \label{1.2}
Assume the conditions \eqref{eq:1.4}, \eqref{eq:1.6}, \eqref{eq:1.7}, \eqref{eq:1.8}, (\ref{eq:1.9}) and \eqref{eq:1.10} are fulfilled. 
Then, given any $\delta>0$ the estimates 
 \begin{equation}\label{eq:1.13}
\left\|\mu\cos(t\sqrt{P})\psi(P^{1/2}/\delta)\mu\right\|+\sum_{\ell=0}^1
\left\|\mu\nabla^\ell P^{-1/2}\sin(t\sqrt{P})\psi(P^{1/2}/\delta)\mu\right\|\le C_0e^{-k(c_0t)},
\end{equation}
\begin{equation}\label{eq:1.14}
\left\|\mu P^{1/2}\sin(t\sqrt{P})\psi(P^{1/2}/\delta)\mu\right\|_1+\sum_{\ell=0}^1
\left\|\mu\nabla^\ell \cos(t\sqrt{P})\psi(P^{1/2}/\delta)\mu\right\|_1\le C_0e^{-k(c_0t)},
\end{equation}
 hold for $t>1$ with constants $c_0,C_0>0$ depending on $\delta$ but independent of $t$. If the dimension $d$ is odd and the condition
 (\ref{eq:1.11}) is assumed, then the estimates (\ref{eq:1.13}) and (\ref{eq:1.14}) hold with $\psi\equiv 1$. 
 \end{Theorem}
 
The present paper extends the exponential local energy decay results obtained in \cite{kn:LSV} to a broader class of electric and magnetic potentials. In \cite{kn:LSV}, the coefficients were assumed to decay exponentially, which permitted analytic continuation of the weighted resolvent, weighted by exponentially decaying weights, and led to exponential decay of the local energy. In this work we replace the exponential weight by the more general weight $\mu(x) = \sqrt{\Theta(c \langle x \rangle)}$, where the function $\Theta$ satisfies the conditions \eqref{eq:1.7} through \eqref{eq:1.10}; note in particular the conditions on the sequence ${m_k}$ given by \eqref{eq:1.9} and \eqref{eq:1.10}. The condition \eqref{eq:1.9} is used in Section \ref{free resolv est section} to obtain the uniform bounds \eqref{eq:2.1} on derivatives of the free weighted resolvent.  The condition \eqref{eq:1.10} ensures stability of these bounds under products and inverses of operator-valued functions, which is used in Section \ref{perturbed resolvent section} for passing from the free resolvent to the perturbed resolvent. The decay rate in time is then determined by the growth of $m_k$ via the above function $k(t)$.

We note that for each $\delta > 0$, cutoff function appearing \eqref{eq:1.13} and \eqref{eq:1.14} is independent of time. This contrasts with the case of exponentially decaying coefficients treated in \cite{kn:LSV}. In that setting, the argument requires bounds on the derivatives of the cutoff of the form $C^{k+1}_\delta k!$ for all all orders $k \le m(t)$, where $C_\delta > 0$ depends on $\delta$ but is independent of $t$, and where the maximum order $m(t)$ of differentiation satisfies $m(t) \to \infty$ as $t \to \infty$. The order $m(t)$ must be finite for each $t$, since there are no nontrivial analytic functions which vanish on an open set.

This paper fits into a long line of work on decay of solutions to the wave equation. We begin by recalling the classical results for the free wave equation, where both $b$ and $V$ are identically zero. If $\mathcal{O} = \emptyset$, the sharp Huygen's principle implies that when $d \ge 3$ is odd, the energy of the solution within any fixed compact set decays to zero in finite time. In even dimensions, the Poisson formula for the solution to the initial value problem implies that the local energy decays like $t^{-d}$. 

When $\mathcal{O} \neq \emptyset$ ($b$ and $V$ still vanishing), the decay of local energy for solutions to \eqref{eq:1.2} is related to the dynamics of the underlying Hamiltonian flow. The non-trapping condition, where all geodesics escape to infinity, is well known to be related to rapid energy decay.  This condition holds for specific geometries, such as convex obstacles and more generally for obstacles where an escape function can be constructed. Foundational works by Lax, Morawetz, and Phillips, Ralston, and Strauss \cite{kn:lmp, kn:M1, kn:M2, kn:MRS} address such scenarios and the resulting decay, utilizing multiplier methods and related properties of the resolvent. In the present work, as well as in \cite{kn:LSV}, the nontrapping assumption is formulated as in \eqref{eq:1.4}.

When nonzero potentials are present, results on energy decay draw from the work of Vainberg \cite{kn:Va} and Melrose-Sj\"ostrand \cite{kn:MS1, kn:MS2}. For example, in the case of a smooth, nonnegative electric potential of compact support, the local energy decay like $O(e^{-Ct})$ for some $C > 0$ when $d \ge 3$ is odd, and like $O(t^{-d})$ when $d \ge 2$ is even. Vainberg \cite{kn:Va, kn:Va1} establishes these decay rates for compactly supported perturbations of the Laplacian satisfying the Generalized Huygens Principle (as defined in \cite{kn:V2a}). Work by Melrose and Sj\"ostrand on the propagation of singularities \cite{kn:MS1, kn:MS2} shows this principle holds for broad class of smooth nontrapping perturbations of the Laplacian. Recent works on energy decay for the wave equation in nontrapping settings are \cite{kn:V1, kn:B1, kn:V1a, kn:V2, kn:BB, kn:CDY, kn:llst25}.


\textbf{Future directions:} A natural direction is to unify the cases a) and b). Achieving this would likely require establishing that Theorem \ref{4.1} holds in the unified setting. We note that if $d = 2$, $\Omega = \emptyset$, and $b$ is nontrivial, \eqref{eq:4.2} holds for $\lambda$ large enough \cite[Theorem 1.1]{kn:MPS}, see also \cite{kn:V3, kn:V5}.  The case of a nontrivial magnetic potential in the presence of a nontrapping obstacle appears to be more delicate. In \cite{kn:LSV}, a resolvent remainder argument is used to transfer the high-frequency bound \eqref{eq:1.4} to the perturbed resolvent; however, this argument does not extend to first-order perturbations. Works relevant to addressing this difficulty include \cite{kn:EGS, kn:CDL, kn:CDL2}.

For medium frequencies, in \cite{kn:LSV}, in the case of an electric potential and a nontrapping obstacle, the proof of \eqref{eq:4.2} relies on a local Carleman estimate from \cite{kn:V4} that imposes Dirichlet boundary condition. Extending this approach to Neumann boundary condition, or to include a magnetic potential, would require establishing a similar local Carleman estimate in more generality.

It would also be interesting to study decay rates for solutions to the wave equation with potentials that decay rapidly as in \eqref{eq:1.6}, but with respect to functions $\Theta$ that do not satisfy \eqref{eq:1.9} or \eqref{eq:1.10}. An example of such a function is $\Theta(r) = e^{-(\log(r + e))^\beta}$, $\beta > 1$.

The rest of the paper is organized as follows. In Section \ref{free resolv est section}, we establish the bound \eqref{eq:2.1} for derivatives of the free weighted resolvent. In Section \ref{perturbed resolvent section}, we derive Theorem \ref{1.1} using \eqref{eq:2.1}, the preliminary estimates proved in Section \ref{prelim estimates section}, along with Theorem \ref{4.1} (proved in \cite{kn:LSV}) and the limiting absorption principle. Section \ref{time decay section} recalls the strategy from \cite{kn:LSV} for converting resolvent estimates into time decay. The appendices collect well-known results on Gevrey cutoff functions, the limiting absorption principle, and the kernel of the free wave propagator, which are used earlier in the paper.

\section{Estimates for the free resolvent} \label{free resolv est section}

Denote by $P_0$ the self-adjoint realization of $-\Delta$
on $L^2(\mathbb{R}^d)$, $d\ge 2$. Let the function $\mu$ and the sequence $m=\{m_k\}_{k=0}^\infty$ be as in the Introduction. 
In this section we will prove the following 

\begin{prop} \label{2.1}
Let $\alpha$ and $\beta$ be multi-indices such that $|\alpha|+|\beta|\le 2$.
 Then, given any $\delta>0$ there is a constant $C=C_\delta>0$ so that the estimate 
\begin{equation}\label{eq:2.1}
\left\|\frac{d^k}{d\lambda^k}\mu\partial_x^\alpha(P_0-\lambda^2)^{-1}\partial_x^\beta\mu\right\|\le C^{k+1}m_k(|\lambda|+1)^{|\alpha|+|\beta|-1}
\end{equation}
holds for all integers $k\ge 0$ and all $\lambda\in\R$, $|\lambda|\ge\delta$. 
  If the dimension $d$ is odd, then the estimate (\ref{eq:2.1}) holds for all $\lambda\in\R$.
 \end{prop}
 
 \begin{proof}
 Since the operator 
$\partial_x^\beta$ commutes with the free resolvent, it suffices to prove the proposition with $\beta=0$ and $|\alpha|\le 2$.
 Denote $\mathbb{C}^- \defeq \{\lambda\in\mathbb{C}:{\rm Im}\,\lambda<0\}$. 
  We will make use of the formula
\begin{equation}\label{eq:2.2}
\lambda(P_0-\lambda^2)^{-1}=i\int_0^\infty e^{-it\lambda}\cos\left(t\sqrt{P_0}\right)dt,\quad \lambda\in\mathbb{C}^-.
\end{equation}
 Differentiating (\ref{eq:2.2}) $k$ times we get
\begin{equation}\label{eq:2.3}
\frac{d^k}{d\lambda^k}\left(\lambda(P_0-\lambda^2)^{-1}\right)=i\int_0^\infty (-it)^ke^{-it\lambda}\cos\left(t\sqrt{P_0}\right)dt
,\quad \lambda\in\mathbb{C}^-.
\end{equation}
Let $\phi\in C_0^\infty(\mathbb{R}^d ; [0,1])$ be such that $\phi(x)=1$ for $|x|\le 1/8$,
$\phi(x)=0$ for $|x|\ge 1/4$. Write
\begin{equation*}
\mu\cos\left(t\sqrt{P_0}\right)\mu=\mathcal{A}_1(t)+\mathcal{A}_2(t),
\end{equation*}
where 
\begin{align*}
& \mathcal{A}_1(t)=\phi(x/t)\mu(x)\cos\left(t\sqrt{P_0}\right)\mu(x)\phi(x/t),\\
& \mathcal{A}_2(t)=(1-\phi(x/t))\mu(x)\cos\left(t\sqrt{P_0}\right)\mu(x)\phi(x/t)
+\mu(x)\cos\left(t\sqrt{P_0}\right)\mu(x)(1-\phi(x/t).
\end{align*}
It is easy to see that for some $c_0 > 0$,
\begin{equation*}
\|\mathcal{A}_2(t)\|\lesssim\sqrt{\Theta(c_0t)}
\end{equation*}
Thus the operator-valued function
\begin{equation*}
\int_0^\infty t^ke^{-it\lambda}\mathcal{A}_2(t)dt
\end{equation*}
extends to $\{{\rm Im}\,\lambda\le 0\}$ and satisfies in this region the bound
\begin{equation}\label{eq:2.4}
\left\|\int_0^\infty t^ke^{-it\lambda}\mathcal{A}_2(t)dt\right\|\lesssim \sup_{t\ge 0}(t+1)^{k+2}\sqrt{\Theta(c_0t)}\le
\widetilde C^{k+1}m_{k+2}\le C^{k+1}m_{k}
\end{equation}
with some constant $C>0$ independent of $k$, where we have used the condition (\ref{eq:1.9}).
On the other hand, if $d$ is odd, 
from the Huygens principle we have that the kernel of the operator
$\cos\left(t\sqrt{P_0}\right)$ vanishes in $|x-y|<t$. Observe now that $x\in {\rm supp}\,\phi(x/t)$ and
$y\in {\rm supp}\,\phi(y/t)$ imply $|x-y|\le t/2$. 
Hence $\mathcal{A}_1(t)\equiv 0$ in this case. Therefore, for odd dimensions it follows from
(\ref{eq:2.3}) and (\ref{eq:2.4}) that (\ref{eq:2.1}) holds with $\alpha=\beta=0$.

Let us see that in even dimensions (\ref{eq:2.1}) with $\alpha=\beta=0$ holds, too. To this end we will use the fact, demonstrated in Appendix \ref{cosine kernel appendix}, that 
when $d$ is even,
the kernel of the operator
$\phi(x/t) \mu(x)\cos\left(t\sqrt{P_0}\right) \mu(x) \phi(x/t)$ is of the form $\mu(x) \mu(y)\phi(x/t)\phi(y/t)t^{-d}W(|x-y|/t)$, where the function $W(z)$ is analytic in $|z|<1$.
Therefore, if $\lambda\in\mathbb{C}^-$, the kernel $\widetilde K$ of the operator
\begin{equation*}
\int_2^\infty t^ke^{-it\lambda}\mathcal{A}_1(t)dt
\end{equation*}
can be written in the form 
\begin{equation*}
\begin{split}
\widetilde K(x,y;\lambda)&=\mu(x)\mu(y)\int_{t_0}^\infty \phi(x/t)\phi(y/t)e^{-it\lambda}t^{k-d}W(|x-y|t^{-1})dt\\
&=\mu(x)\mu(y)\int_{t_0}^\infty e^{-it\lambda}t^{k-d}W(|x-y|t^{-1})dt\\
&-\mu(x)\mu(y)\int_{t_0}^\infty ((1-\phi)(x/t)\phi(y/t)
+(1-\phi)(y/t))e^{-it\lambda}t^{k-d}W(|x-y|t^{-1})dt\\
&=: K_1(\lambda)+K_2(\lambda),
\end{split}
\end{equation*}
where $t_0 = 2\max\{1,|x-y|\}$. Note that the function $W(|x-y|t^{-1})$ extends to complex $t$ and 
\begin{equation*}
|W(|x-y|t^{-1})|\lesssim 1\quad\mbox{for}\quad|t|\ge t_0.
\end{equation*}
Observe that
\begin{equation*}
\mu(x)^{1/2}(1-\phi)(x/t)\lesssim \Theta(c_1t)^{1/4},\,c_1>0,
\end{equation*}
 and
\begin{equation*}
\sup_{t\ge 2}t^k\Theta(c_1t)^{1/4}=\left(\sup_{t\ge 2}t^{2k}\sqrt{\Theta(c_1t)}\right)^{1/2}
\le \widetilde C^{k+1}m_{2k}^{1/2}\le C^{k+1}m_k,
\end{equation*}
where we have used the condition (\ref{eq:1.9}).
 Hence 
the function $K_2(\lambda)$ can be extended to $\{{\rm Im}\,\lambda\le 0\}$ and satisfies in this region the bound
\begin{equation}\label{eq:2.5}
|K_2(\lambda)|\lesssim C^{k+1}m_k\mu(x)^{1/2}\mu(y)^{1/2}\int_2^\infty t^{-d}dt\lesssim 
C^{k}m_k\mu(x)^{1/2}\mu(y)^{1/2}.
\end{equation}
 We will now bound $|K_1|$ for $\lambda\in \R$, $\lambda\ge\delta$, $\delta>0$ being arbitrary. To this end, 
 observe that we can change the contour of integration in the integral, which
 we denote by $I(\lambda)$, and write it in the form
 \begin{equation*}
 I(\lambda)=t_0^{k-d+1} e^{-i \theta}\int_{0}^\infty e^{-it_0(1+\sigma e^{-i\theta})\lambda}
 (1+\sigma e^{-i\theta})^{k-d}W(|x-y|(t_0(1+\sigma e^{-i\theta}))^{-1})d\sigma
 \end{equation*}
 where $0<\theta\ll 1$ is a constant. Then 
 $I(\lambda)$ satisfies the bounds
 \begin{equation*}
 \begin{split}
 |I(\lambda)|&\lesssim t_0^{k}\int_{0}^\infty e^{-\sigma t_0\delta\sin\theta}
 |1+\sigma e^{-i\theta}|^{k-d}d\sigma\\
 &\lesssim t_0^{k}\sup_{\sigma>0}(\sigma+1)^ke^{-\sigma t_0\delta\sin\theta}\int_0^\infty(\sigma+1)^{-d}d\sigma\\
 &\lesssim (2t_0)^{k}+\sup_{\sigma\ge 1}(2t_0\sigma)^ke^{-\sigma t_0\delta\sin\theta}\\
  &\lesssim (2t_0)^{k}+(2/\delta\sin\theta)^kk!.
 \end{split}
 \end{equation*}
 Since $t_0\sim \langle x-y\rangle$, this implies 
 \begin{equation}\label{eq:2.6}
 |K_1(\lambda)|\lesssim C^k\left(\langle x-y\rangle^{k}+k!\right)\mu(x)\mu(y)
 \end{equation}
 for $\lambda\ge\delta$, with some constant $C>0$ independent of $k$ but depending on $\delta$. 
 On the other hand, in view of the condition (\ref{eq:1.8}), we have 
  \begin{equation}\label{eq:2.7}
  \mu(x)^{1/2}\mu(y)^{1/2}\lesssim \Theta(c_2\langle x-y\rangle)^{1/4},\quad c_2>0.
   \end{equation}
   By (\ref{eq:2.7}),
   \begin{equation}\label{eq:2.8}
   \langle x-y\rangle^{k}\mu(x)^{1/2}\mu(y)^{1/2}\lesssim \left(\sup_{r\ge 1}r^{2k}\sqrt{\Theta(c_2r)}\right)^{1/2}\le \widetilde C^{k+1}m_{2k}^{1/2}\le C^{k+1}m_k,
    \end{equation}
    where we have used the condition (\ref{eq:1.9}). 
     By (\ref{eq:2.6}) and (\ref{eq:2.8}),
     \begin{equation}\label{eq:2.9}
 |K_1(\lambda)|\lesssim C^km_k\mu(x)^{1/2}\mu(y)^{1/2}.
 \end{equation}
 By (\ref{eq:2.5}) and (\ref{eq:2.9}), we obtain
 \begin{equation}\label{eq:2.10}
 |\widetilde K(x,y;\lambda)|\lesssim C^km_k\mu(x)^{1/2}\mu(y)^{1/2}.
 \end{equation}
 It follows from (\ref{eq:2.10}) and Schur's lemma that the operator with kernel $\widetilde K$
is bounded on $L^2$ uniformly in $\lambda\ge\delta$ with a bound $O\left(C^km_k\right)$.
It is easy to see that the same conclusion holds for $\lambda\le -\delta$.
Furthermore, it is clear that the operator-valued function
\begin{equation*}
\int_0^2 t^ke^{-it\lambda}\mathcal{A}_1(t)dt
\end{equation*}
extends to $\{{\rm Im}\,\lambda\le 0\}$ and its norm is $O\left(2^k\right)$.  Thus we conclude that an analog of the estimate 
 (\ref{eq:2.4}) with the operator $\mathcal{A}_2$ replaced by $\mathcal{A}_1$ holds for $\lambda\in\R$, $|\lambda|\ge\delta$, 
 with a new constant $C>0$ depending on $\delta$. 
 This implies (\ref{eq:2.1}) with $\alpha=\beta=0$ in this case.
 
 We will now show that (\ref{eq:2.1}) with $\alpha=\beta=0$ implies (\ref{eq:2.1}) with $1\le |\alpha|\le 2$ and $\beta=0$. 
 It suffices to do so for $k\ge 1$ since for $k=0$ the estimate is well-known. 
 To this end we need the following
 
 \begin{lemma} \label{2.2} 
Let $\alpha$ be a multi-index such that $|\alpha|\le 2$. 
 Then, there exist constants $C > 0$ and $\theta_0 > 0$ such that for all $\theta \ge \theta_0$ we have the estimate
\begin{equation}\label{eq:2.11}
\left\|\mu\partial_x^\alpha\left(P_0-i\theta^2\right)^{-1}\mu^{-1}
\right\|_{L^2\to L^2}\le C \theta^{|\alpha|-2}.
\end{equation}
\end{lemma}

Differentiating the resolvent identity
 \begin{equation*}
 (P_0-\lambda^2)^{-1}=(P_0-i\theta^2)^{-1}+(\lambda^2-i\theta^2)(P_0-i\theta^2)^{-1}(P_0-\lambda^2)^{-1},\quad \lambda\in\mathbb{C}^-,
 \end{equation*}
 we obtain
 \begin{equation*}
 \begin{split}
 \frac{d^k}{d\lambda^k}(P_0-\lambda^2)^{-1}&=(\lambda^2-i\theta^2)(P_0-i\theta^2)^{-1}\frac{d^k}{d\lambda^k}(P_0-\lambda^2)^{-1}\\
 &+2k\lambda(P_0-i\theta^2)^{-1}\frac{d^{k-1}}{d\lambda^{k-1}}(P_0-\lambda^2)^{-1}\\
 &+k(k-1)(P_0-i\theta^2)^{-1}\frac{d^{k-2}}{d\lambda^{k-2}}(P_0-\lambda^2)^{-1}
  \end{split}
 \end{equation*}
 for all integers $k\ge 1$. Thus we get the identity
 \begin{equation*}
 \begin{split}
 \frac{d^k}{d\lambda^k}\mu\partial_x^\alpha(P_0-\lambda^2)^{-1}\mu&=(\lambda^2-i\theta^2)\mu\partial_x^\alpha(P_0-i\theta^2)^{-1}\mu^{-1}
 \frac{d^k}{d\lambda^k}\mu(P_0-\lambda^2)^{-1}\mu\\
 &+2k\lambda\mu\partial_x^\alpha(P_0-i\theta^2)^{-1}\mu^{-1}\frac{d^{k-1}}{d\lambda^{k-1}}\mu(P_0-\lambda^2)^{-1}\mu\\
 &+k(k-1)\mu\partial_x^\alpha(P_0-i\theta^2)^{-1}\mu^{-1}\frac{d^{k-2}}{d\lambda^{k-2}}\mu(P_0-\lambda^2)^{-1}\mu
  \end{split}
 \end{equation*}
 for all integers $k\ge 1$, which clearly extends to $\{{\rm Im}\,\lambda\le 0\}$. By (\ref{eq:2.1}) with $\alpha=\beta=0$
 and (\ref{eq:2.11}),
 \begin{equation}\label{eq:2.12}
 \begin{split}
 \left\|\frac{d^k}{d\lambda^k}\mu\partial_x^\alpha
 (P_0-\lambda^2)^{-1}\mu\right\|&\lesssim\theta^{|\alpha|-2}(|\lambda|^2+\theta^2)(|\lambda|+1)^{-1}C^{k+1}m_k\\
 &+\theta^{|\alpha|-2}C^{k}km_{k-1}+\theta^{|\alpha|-2}C^{k-1}k(k-1)m_{k-2}
  \end{split}
 \end{equation}
 for all $\lambda\in\R$, $|\lambda|\ge\delta$, and all $\theta\ge\theta_0$. We now take $\theta=\theta_0+|\lambda|$.
 It easy to see that (\ref{eq:2.12}) implies (\ref{eq:2.1}) with 
 $1\le |\alpha|\le 2$ and $\beta=0$. 
 In the same way, in odd dimensions we get that (\ref{eq:2.1}) with $\alpha=\beta=0$ and all $\lambda\in\R$ implies (\ref{eq:2.1}) with 
 $1\le |\alpha|\le 2$, $\beta=0$ 
 and all $\lambda\in\R$. 
 
 It remains to prove the estimate (\ref{eq:2.1}) with $\alpha=\beta=0$ for $|\lambda|\le 1$ in odd dimensions $d\ge 3$. To this end, we will use
 that the kernel $K_0(x,y;\lambda)$ of the resolvent $(P_0-\lambda^2)^{-1}$, $\lambda\in\mathbb{C}^-$,
can be expressed in terms of the Henkel functions by the formula
\begin{equation*}
K_0(x,y;\lambda)=i2^{-2}(2\pi)^{-\frac{d-2}{2}}\lambda^{\frac{d-2}{2}}|x-y|^{-\frac{d-2}{2}}H^-_{\frac{d-2}{2}}(\lambda|x-y|).
\end{equation*}
It is well-known that $z^{\frac{d-2}{2}}H^-_{\frac{d-2}{2}}(z)=p_d(z)e^{-iz}$, where $p_d(z)$ is a polynomial of order
 $\frac{d-3}{2}$. We have the formula
  \begin{equation*}
  \partial_z^k(p_d(z)e^{iz})=\sum_{\nu=0}^{\frac{d-3}{2}}\frac{k!}{\nu!(k-\nu)!}\frac{d^\nu p_d(z)}{dz^\nu}(-i)^{k-\nu}e^{-iz}
 \end{equation*}
 which imlpies the bounds 
\begin{equation}\label{eq:2.13}
\left|\partial_z^k\left(z^{\frac{d-2}{2}}H^-_{\frac{d-2}{2}}(z)\right)\right|\lesssim 
\sum_{\nu=0}^{\frac{d-3}{2}}k^\nu(|z|+1)^{\frac{d-3}{2}-\nu}\lesssim |z|^{\frac{d-3}{2}}+(k+1)^{\frac{d-3}{2}}
\end{equation}
 for ${\rm Im}\,z\le 0$ and all integers $k\ge 0$. By (\ref{eq:2.13}),
 \begin{equation}\label{eq:2.14}
\left|\partial_\lambda^kK_0(x,y;\lambda)\right|\lesssim|x-y|^{k-\frac{d-1}{2}}+(k+1)^{\frac{d-3}{2}}|x-y|^{k-d+2}
\end{equation}
 for ${\rm Im}\,\lambda\le 0$, $|\lambda|\le 1$, and all integers $k\ge 0$. By (\ref{eq:2.14}) and (\ref{eq:2.8}),
 \begin{equation}\label{eq:2.15}
\left|\mu(x)\mu(y)\partial_\lambda^kK_0(x,y;\lambda)\right|\le C^{k+1}m_k\left(|x-y|^{-\frac{d-1}{2}}+|x-y|^{-d+2}\right)
\mu(x)^{1/2}\mu(y)^{1/2}.
\end{equation}
 Clearly, (\ref{eq:2.1}) with $|\alpha| = |\beta| = 0$ and $|\lambda|\le 1$ 
 follows from (\ref{eq:2.15}) and Schur's lemma. \\
 \end{proof}
 
 \begin{proof}[Proof of Lemma \ref{2.2}] 
We utilize the standard resolvent estimate: there exists $C > 0$ so that for any multi-indices $\alpha, \beta$ with $|\alpha| + |\beta| \le 2$, and for any $\theta \ge 1$, 
\begin{equation} \label{free resolv bd}
\| \partial_x^\alpha (P_0 - i\theta^2)^{-1} \partial_x^\beta \|_{L^2 \to L^2} \le C \theta^{|\alpha| + |\beta| - 2}.
\end{equation}
 We compute the commutator
\begin{equation*}
\begin{split}
\mu(P_0 - i \theta^2) - (P_0 - i\theta^2) \mu &= - \Delta \mu -2 (\nabla \mu) \cdot \nabla  \\
&=- \sum_{j =1} \partial_{x_j} (\partial_{x_j} \mu) - \Delta \mu \qefed [ \mu, P_0],
\end{split}
\end{equation*}
which implies the resolvent identity
\begin{equation*}
\mu (P_0 - i\theta^2)^{-1} = (P_0 - i \theta^2)^{-1}  \mu -  (P_0 - i \theta^2)^{-1}  [ \mu, P_0]  (P_0 - i \theta^2)^{-1} .
\end{equation*}
Therefore, for $f \in C_0^\infty(\R^d)$, 
\begin{equation} \label{resolv id with f}
\mu(P_0 - i \theta^2)^{-1} \mu^{-1} f = (P_0 - \theta^2)^{-1} f -  (P_0 - \theta^2)^{-1} [\mu, P_0] (P_0 - i\theta^2)^{-1} \mu^{-1} f.
\end{equation}
Continuing, for any $u \in C_0^\infty(\R^d)$, 
\begin{equation} \label{commutator insert weights}
[\mu, P_0] u = \left[ -(\Delta \mu) \mu^{-1} - \sum_{j=1}^d \partial_{x_j} \big((\partial_{x_j} \mu) \mu^{-1} \big)  \right] (\mu u)
\end{equation}
By \eqref{eq:1.7}, $(\partial_{x_j} \mu)\mu^{-1}$ and $(\Delta \mu) \mu^{-1}$ are bounded multiplication operators on $L^2$. So \eqref{commutator insert weights} continues to hold for any $u \in L^2(\R^d)$, with both sides being interpreted as members of the negative index Sobolev space $H^{-1}(\R^d)$. Therefore, by \eqref{free resolv bd}, \eqref{resolv id with f}, and \eqref{commutator insert weights}, 

\begin{equation*}
\begin{split}
\| \mu (P_0 - i\theta^2)^{-1} \mu^{-1} f \|_{L^2} &\le \| (P_0 - i\theta^2)^{-1} f \|_{L^2} + \| (P_0 - i\theta^2)^{-1} \|_{H^{-1} \to L^2} \| \mu (P_0 - i\theta^2)^{-1} \mu^{-1} f \|_{L^2} \\
& \le C(\theta^{-2} \| f\|_{L^2} + \theta^{-1} \| \mu (P_0 - i\theta^2)^{-1} \mu^{-1} f \|_{L^2}),
\end{split} 
\end{equation*}
for some $C > 0$ independent of $\theta$ and $f$. Thus taking $\theta$ large enough we can absorb the second term on the right side into the left hand side, completing the proof of \eqref{eq:2.11} when $|\alpha| = 0$. Standard arguments using elliptic regularity then imply the general case  $|\alpha | \le 2$.\\
 \end{proof}
 
 \section{Preliminary estimates}  \label{prelim estimates section}
 
Let $\mathcal{L}(L^2)$ denote the set of the bounded operators on $L^2(\Omega)$. Let $\Lambda\subset\R$ be a closed interval.
Let also $\{m_k\}_{k= 0}^\infty$ be the sequence introduced in Section 1. 
Our goal in this section is to prove the following two lemmas.

\begin{lemma}  \label{3.1} 
Let $\mathcal{M}_j(\lambda)\in C^\infty(\Lambda;\mathcal{L}(L^2))$, $j=1, \dots ,J$, be operator-valued functions 
satisfying the bounds
\begin{equation}\label{eq:3.1}
\left\|\frac{d^k\mathcal{M}_j(\lambda)}{d\lambda^k}\right\|\le M_j(\lambda)C_j^km_k,\quad\lambda\in\Lambda,
\end{equation}
for all integers $k\ge 0$, with constants $C_j>0$ independent of $k$ and $\lambda$, and functions $M_j>0$ independent of $k$.
Then the operator
\begin{equation*}
\mathcal{M}(\lambda)= \prod_{j= 1}^J\mathcal{M}_j(\lambda) \in C^\infty(\Lambda;\mathcal{L}(L^2))
\end{equation*}
satisfies the bound
\begin{equation}\label{eq:3.2}
\left\|\frac{d^k\mathcal{M}(\lambda)}{d\lambda^k}\right\|\le M(\lambda)C^km_k,\quad\lambda\in\Lambda,
\end{equation}
for all integers $k\ge 0$, with a constant $C>0$ independent of $k$ and $\lambda$, and where
\begin{equation*}
M(\lambda)=\prod_{j=1}^JM_j(\lambda).
\end{equation*}
\end{lemma}

\begin{proof} 
Clearly, it suffices to prove the lemma with $J=2$. 
By the Leibnitz formula we have
\begin{equation}\label{eq:3.3}
\frac{d^k}{d\lambda^k}(\mathcal{M}_1(\lambda)\mathcal{M}_2(\lambda))=
\sum_{\nu=0}^k\frac{k!}{\nu!(k-\nu)!}\frac{d^\nu \mathcal{M}_1(\lambda)}{d\lambda^\nu}
\frac{d^{k-\nu}\mathcal{M}_2(\lambda)}{d\lambda^{k-\nu}}.
\end{equation}
By \eqref{eq:3.1} and \eqref{eq:3.3} together with the condition \eqref{eq:1.10} we get
\begin{equation}\label{eq:3.4}
\begin{split}
\left\|\frac{d^k}{d\lambda^k}(\mathcal{M}_1(\lambda)\mathcal{M}_2(\lambda))\right\|
&\le M_1(\lambda)M_2(\lambda)\sum_{\nu=0}^k\frac{k!}{\nu!(k-\nu)!}C_1^{\nu+1}m_\nu C_2^{k-\nu+1}m_{k-\nu}\\
&\le M_1(\lambda)M_2(\lambda)C^{k+1}m_{k}
\end{split}
\end{equation}
for every integer $k \ge 0$ with some constant $C>0$ independent of $k$ and $\lambda$, which confirms \eqref{eq:3.2} in this case. \\
\end{proof}

\begin{lemma}  \label{3.2} 
Let $\mathcal{M}(\lambda)\in C^\infty(\Lambda;\mathcal{L}(L^2))$ be an operator-valued function 
satisfying the bounds
\begin{equation}\label{eq:3.5}
\left\|\frac{d^k\mathcal{M}(\lambda)}{d\lambda^k}\right\|\le C^{k+1}m_k,\quad\lambda\in\Lambda,
\end{equation}
for all integers $k\ge 0$ with a constant $C>0$ independent of $k$ and $\lambda$. Suppose also that $\mathcal{M}(\lambda)$ is invertible 
with an inverse satisfying the bound
\begin{equation}\label{eq:3.6}
\left\|\mathcal{M}^{-1}(\lambda)\right\|\le \widetilde C,\quad\lambda\in\Lambda,
\end{equation}
 with a constant $\widetilde C>0$ independent of $\lambda$. Set $B = \max(2C, 2C \tilde C C_4)$, where $C_4$ is as in \eqref{eq:1.10}. Then the operator $\mathcal{M}^{-1}(\lambda) \in C^\infty(\Lambda;\mathcal{L}(L^2))$
satisfies the bound
\begin{equation}\label{eq:3.7}
\left\|\frac{d^k\mathcal{M}^{-1}(\lambda)}{d\lambda^k}\right\|\le B^{k+1}m_k,\quad\lambda\in\Lambda,
\end{equation}
for all integers $k\ge 0$ with a constant $B>0$ independent of $k$ and $\lambda$. 
\end{lemma}

\begin{proof} 
Differentiating the identity
\begin{equation*}
\mathcal{M}(\lambda)\mathcal{M}^{-1}(\lambda)=I
 \end{equation*}
 we get
 \begin{equation}\label{eq:3.8}
 \mathcal{M}(\lambda)\frac{d^{k+1}\mathcal{M}^{-1}(\lambda)}{d\lambda^{k+1}}=-\sum_{\nu=0}^k\frac{(k+1)!}{\nu!(k+1-\nu)!}\frac{d^{k+1-\nu} \mathcal{M}(\lambda)}{d\lambda^{k+1-\nu}}
\frac{d^{\nu}\mathcal{M}^{-1}(\lambda)}{d\lambda^{\nu}}
 \end{equation}
 for any integer $k\ge 0$. By \eqref{eq:3.5}, \eqref{eq:3.6} and \eqref{eq:3.8},
 \begin{equation}\label{eq:3.9}
 \left\|\frac{d^{k+1}\mathcal{M}^{-1}(\lambda)}{d\lambda^{k+1}}\right\|\le \tilde C \sum_{\nu=0}^kC^{k+1-\nu}\frac{(k+1)! m_{k+1-\nu}}{\nu!(k+1-\nu)!}
\left\|\frac{d^{\nu}\mathcal{M}^{-1}(\lambda)}{d\lambda^{\nu}}\right\|.
 \end{equation}
 Now \eqref{eq:3.7} follows from \eqref{eq:3.9} by induction in $k$. Clearly, \eqref{eq:3.7} with $k=0$ follows from \eqref{eq:3.6},
 provided $B\ge\widetilde C$. 
 Furthermore, if we suppose that \eqref{eq:3.7} holds for all $\nu\le k$, by
 \eqref{eq:3.9} and \eqref{eq:1.10}, we obtain 
 \begin{equation}\label{eq:3.10}
 \begin{split}
 \left\|\frac{d^{k+1}\mathcal{M}^{-1}(\lambda)}{d\lambda^{k+1}}\right\|&\le \tilde C
 \sum_{\nu=0}^kC^{k+1-\nu}B^{\nu+1}\frac{(k+1)!m_\nu m_{k+1-\nu}}{\nu!(k+1-\nu)!}\\
 &\le  C\tilde C C_4 B^{k+1}m_{k+1} \sum_{\nu=0}^k \left( \frac{C}{B} \right)^{k-\nu}\\
  &\le 2 C\tilde C C_4 B^{k+1}m_{k+1} \le B^{k+2}m_{k+1}
 \end{split}
\end{equation}
Observe that we used $\sum_{\nu = 0}^{k} (C/B)^{k - \nu} \le 2$ (since $B\ge 2C$) and $B \ge 2C \tilde C C_4$. Thus, we conclude from 
\eqref{eq:3.10} that \eqref{eq:3.7} holds for $k+1$ as well, as desired. \\
 \end{proof}
 
\section{Estimates for the perturbed resolvent} \label{perturbed resolvent section}

Let $P$ be the self-adjoint operator from Section 1 and suppose that the electric potential $V\in L^\infty(\Omega,\mathbb{R})$
and the magnetic potential 
$b\in L^\infty(\Omega,\mathbb{R}^d)$ satisfy
\begin{equation}\label{eq:4.1}
|V(x)|+|b(x)|\le C\langle x\rangle^{-\rho}, \quad C>0,\,\rho>1.
\end{equation}
The following weighted resolvent bounds for $P$ are proved in \cite{kn:LSV}.

\begin{Theorem} \label{4.1}
Let $s > 1/2$. Assume the conditions \eqref{eq:1.4} and \eqref{eq:4.1} are fulfilled. Then, given any $\delta>0$, there is $C_\delta>0$ such that 
\begin{equation}\label{eq:4.2}
\left\|\langle x\rangle^{-s}\partial_x^\alpha(P-\lambda^2\pm i\varepsilon)^{-1}\partial_x^\beta\langle x\rangle^{-s}\right\|
\le C_\delta\lambda^{|\alpha|+|\beta|-1},\quad\lambda\ge\delta,\,0<\varepsilon<1,
\end{equation}
where $\alpha$ and $\beta$ are multi-indices such that $|\alpha|, \, |\beta| \le 1$.  
\end{Theorem}

In Appendix \ref{limiting absorption appendix} we recall how \eqref{eq:4.2} implies the limiting absorption principle, namely that  
\begin{equation*}
(P-\lambda^2)^{-1} : \langle x\rangle^{-s}H^{-1}(\mathbb{R}^d)\to \langle x\rangle^{s}H^1(\mathbb{R}^d)
\end{equation*}
has a continuous extension from $\pm \Im \lambda > 0$ to $\R \setminus \{0\}$. In what follows in this section we derive Theorem \ref{1.1} from Proposition \ref{2.1}, Theorem \ref{4.1}, and the limiting absorption principle.

 \begin{proof}[Proof of Theorem \ref{1.1} in case a)]
 We will make use of the resolvent identities obtained in \cite[Section 3]{kn:LSV}, which allow to express the perturbed resolvent in terms
 of the free one. We will keep the same notation. Let $\lambda\in\mathbb{C}^-$ and denote by $I$ the identity operator. 
 The analysis in \cite{kn:LSV} is based on the following two resolvent identities
\begin{align*}
& (P- \lambda^2)^{-1} (\widetilde{V} +  i\nabla\cdot b+ib\cdot\nabla) =  
I - (P- \lambda^2)^{-1} (P_0 - \lambda^2) \quad \text{on }H^2(\R^d), \\
& (P_0- \lambda^2)^{-1} (\widetilde{V} +  i\nabla\cdot b+ib\cdot\nabla) =  
-I + (P_0- \lambda^2)^{-1} (P - \lambda^2) \quad \text{on } D(P), 
\end{align*}
 where $\widetilde V = V + |b|^2$. These yield
\begin{equation}\label{eq:4.3}
\begin{split}
&(P-\lambda^2)^{-1}-(P_0-\lambda^2)^{-1}\\
&=-(P_0-\lambda^2)^{-1}(\widetilde V+i\nabla\cdot b+ib\cdot\nabla)(P-\lambda^2)^{-1}\\
&=-(P-\lambda^2)^{-1}(\widetilde V+i\nabla\cdot b+ib\cdot\nabla)(P_0-\lambda^2)^{-1}.
\end{split}
\end{equation}
Let $z\in\mathbb{C}^-$. By \eqref{eq:4.3},
\begin{equation}\label{eq:4.4}
\begin{split}
&(P-\lambda^2)^{-1}-(P-z^2)^{-1}\\
&=(\lambda^2-z^2)(P-z^2)^{-1}(P-\lambda^2)^{-1}\\
&=L^\sharp(z)((P_0-\lambda^2)^{-1}-(P_0-z^2)^{-1})L^\flat(\lambda),
\end{split}
\end{equation}
where
\begin{align*}
& L^\sharp(z)=I-(P-z^2)^{-1}(\widetilde V+i\nabla\cdot b+ib\cdot\nabla),\\
 & L^\flat(\lambda) =I-(\widetilde V+i\nabla\cdot b+ib\cdot\nabla)(P-\lambda^2)^{-1}.
 \end{align*}
Multiplying (\ref{eq:4.4}) on left by $\mu \nabla^{\ell}$, $\ell \in \{0, 1\}$, and the right by $\mu$, we get
\begin{equation}\label{eq:4.5}
\begin{split}
&\mu \nabla^{\ell} (P-\lambda^2)^{-1}\mu-\mu \nabla^{\ell}(P-z^2)^{-1}\mu\\
&=\sum_{\ell_1=0}^1\sum_{\ell_2=0}^1L^\sharp_{\ell, \ell_1}(z)\mu^{1-\ell_1}(-i\mu^{-1}b\cdot\nabla)^{\ell_1}((P_0-\lambda^2)^{-1}
-(P_0-z^2)^{-1})(-i\nabla\cdot b\mu^{-1})^{\ell_2}\mu^{1-\ell_2}L^\flat_{\ell_2}(\lambda)
\end{split}
\end{equation}
where, for $\ell \in \{0, 1\}$,
\begin{align*}
& L_{\ell,0}^\sharp= (\mu \nabla \mu^{-1})^{\ell} -\mu \nabla^{\ell}(P-z^2)^{-1}(\widetilde V+i\nabla\cdot b)\mu^{-1},\\
& L_{\ell,1}^\sharp=\mu \nabla^{\ell}(P-z^2)^{-1}\mu,\\
& L_0^\flat=I-\mu^{-1}(\widetilde V+ib\cdot\nabla)(P-\lambda^2)^{-1}\mu,\\
& L_1^\flat=\mu(P-\lambda^2)^{-1}\mu,
\end{align*}
We now let the operator $i\mu^{-1}b\cdot\nabla$ act on the left side of (\ref{eq:4.4}) and multiply the right side by $\mu$.
We get
\begin{equation}\label{eq:4.6}
i\mu^{-1}b\cdot\nabla(P-\lambda^2)^{-1}\mu=T_1(\lambda, z)+T_2(\lambda,z)\mu(P-\lambda^2)^{-1}\mu+T_3(\lambda,z)
\mu^{-1}ib\cdot\nabla(P-\lambda^2)^{-1}\mu,
\end{equation}
where
\begin{align*}
& T_1=i\mu^{-1}b\cdot\nabla(P-z^2)^{-1}\mu -\sum_{\ell_1=0}^1 \widetilde L^\sharp_{\ell_1}(z)\mu^{1-\ell_1}
(i\mu^{-1}b\cdot\nabla)^{\ell_1}((P_0 - \lambda^2)^{-1} - (P_0 - z^2)^{-1})\mu ,\\
& T_2=\sum_{\ell_1=0}^1\sum_{\ell_2=0}^1\widetilde L^\sharp_{\ell_1}(z)\mu^{1-\ell_1}(i\mu^{-1}b\cdot\nabla)^{\ell_1}((P_0-\lambda^2)^{-1}-(P_0-z^2)^{-1})
(\widetilde V\mu^{-1})^{1-\ell_2}(i\nabla\cdot b\mu^{-1})^{\ell_2},\\
& T_3=\sum_{\ell_1=0}^1\widetilde L^\sharp_{\ell_1}(z)\mu^{1-\ell_1}
(i\mu^{-1}b\cdot\nabla)^{\ell_1}((P_0-\lambda^2)^{-1}-(P_0-z^2)^{-1})\mu,\\
&\widetilde L_0^\sharp=i\mu^{-1}b\cdot\nabla(P-z^2)^{-1}(\widetilde V+i\nabla\cdot b)\mu^{-1},\\
&\widetilde L_1^\sharp=-I+i\mu^{-1}b\cdot\nabla(P-z^2)^{-1}\mu.
\end{align*}
In view of the limiting absorption principle, the identities (\ref{eq:4.5}) and (\ref{eq:4.6}) extend to all $\lambda,z\in\R\setminus \{0\}$. 
Fix $z \in \R$, $|z|\ge\delta$, with $0<\delta\ll 1$ being arbitrary. Consider the above operators as functions of $\lambda\in\Lambda_\gamma(z):=\{\lambda\in\R:|\lambda-z|\le\gamma\}$, where $0<\gamma<\delta/2$ is a small parameter
independent of $z$ to be fixed below. Due to the condition \eqref{eq:1.6}, the operators 
$\widetilde L_{0,0}^\sharp$ and $\widetilde L_{0,1}^\sharp$ are bounded on $L^2(\R^d)$. Moreover, it follows from the resolvent estimate (\ref{eq:4.2}) that
\begin{equation}\label{eq:4.7}
\|\widetilde L_{\ell_1}^\sharp(z)\|\lesssim (|z|+1)^{1-\ell},\quad \ell_1=0,1.
\end{equation}
On the other hand, using the estimate (\ref{eq:2.1}) with $k=1$,
\begin{equation}\label{eq:4.8}
\begin{split}
&\left\|\mu \nabla^{\ell}((P_0-\lambda^2)^{-1}-(P_0-z^2)^{-1})\mu\right\|\\
&\lesssim|\lambda-z|\sup_{\sigma=\kappa\lambda+(1-\kappa)z,\,0\le\kappa\le 1}\left\|\frac{d}{d\sigma} \mu \nabla^{\ell}(P_0-\sigma^2)^{-1}\mu\right\|\\
&\lesssim |\lambda-z|(|z|+1)^{\ell-1},\quad \ell=0,1.
\end{split}
\end{equation}
By (\ref{eq:4.7}) and (\ref{eq:4.8}),
\begin{equation}\label{eq:4.9}
\|T_3(\lambda,z)\|\lesssim \gamma\le 1/2,
\end{equation}
if $\lambda\in\Lambda_\gamma(z)$ with a constant $\gamma>0$ small enough. Therefore, the operator $I-T_3$
is invertible for $\lambda\in\Lambda_\gamma(z)$, so we can rewrite (\ref{eq:4.6}) in the form
\begin{equation}\label{eq:4.10}
\mu^{-1}ib\cdot\nabla(P-\lambda^2)^{-1}\mu=(I-T_3)^{-1}T_1+(I-T_3)^{-1}T_2\mu(P-\lambda^2)^{-1}\mu.
\end{equation}
By (\ref{eq:4.5}) with $\ell = 0$ and (\ref{eq:4.10}),
\begin{equation}\label{eq:4.11}
\mu(P-\lambda^2)^{-1}\mu=F_1(\lambda,z)+F_2(\lambda,z)\mu(P-\lambda^2)^{-1}\mu,\quad \lambda\in\Lambda_\gamma(z),
\end{equation}
where
\begin{align*}
 F_1&=\mu(P-z^2)^{-1}\mu + \sum_{\ell_1 = 0}^1 L^\sharp_{0, \ell_1}(z)\mu^{1-\ell_1}(-i\mu^{-1}b\cdot\nabla)^{\ell_1} ((P_0-\lambda^2)^{-1}-(P_0-z^2)^{-1})\mu \\
 &-\sum_{\ell_1=0}^1L^\sharp_{0, \ell_1}(z)\mu^{1-\ell_1}(-i\mu^{-1}b\cdot\nabla)^{\ell_1}
((P_0-\lambda^2)^{-1}-(P_0-z^2)^{-1})\mu(I-T_3)^{-1}T_1,\\
 F_2&=\sum_{\ell_1=0}^1\sum_{\ell_2=0}^1L^\sharp_{0, \ell_1}(z)\mu^{1-\ell_1}(-i\mu^{-1}b\cdot\nabla)^{\ell_1} \\
 &\cdot ((P_0-\lambda^2)^{-1}-(P_0-z^2)^{-1})(-i\nabla\cdot b\mu^{-1})^{\ell_2}(-\widetilde V\mu^{-1})^{1-\ell_2}\\
&-\sum_{\ell_1=0}^1L^\sharp_{0, \ell_1}(z)\mu^{1-\ell_1}(-i\mu^{-1}b\cdot\nabla)^{\ell_1}
((P_0-\lambda^2)^{-1}-(P_0-z^2)^{-1})\mu(I-T_3)^{-1}T_2.
\end{align*}
 By (\ref{eq:4.2}) we also have
\begin{equation}\label{eq:4.12}
\|T_j(\lambda,z)\|\lesssim (|z|+1)^{j-1},\quad j=1,2,
\end{equation}
\begin{equation}\label{eq:4.13}
\|L_{0, \ell_1}^\sharp(z)\|\lesssim (|z|+1)^{-\ell_1},\quad \ell_1=0,1.
\end{equation}
As above, by (\ref{eq:4.9}), (\ref{eq:4.12}), (\ref{eq:4.13}) and (\ref{eq:2.1}) with $k=1$,
\begin{equation}\label{eq:4.14}
\|F_1(\lambda,z)\|\lesssim (|z|+1)^{-1},
\end{equation}
\begin{equation}\label{eq:4.15}
\|F_2(\lambda,z)\|\lesssim \gamma\le 1/2,
\end{equation}
if $\lambda\in\Lambda_\gamma(z)$ with possibly a new constant $\gamma>0$ small enough.
Hence, the operator $I-F_2$
is invertible for $\lambda\in\Lambda_\gamma(z)$, and we get the identity
\begin{equation}\label{eq:4.16}
\mu(P-\lambda^2)^{-1}\mu=(I-F_2(\lambda,z))^{-1}F_1(\lambda,z),\quad \lambda\in\Lambda_\gamma(z).
\end{equation}
We will use (\ref{eq:4.16}) to bound the norm of the operator
\begin{equation*}
\frac{d^k}{d\lambda^k}\mu(P-\lambda^2)^{-1}\mu,\quad \lambda\in\Lambda_\gamma(z),
\end{equation*}
where $k\ge 1$ is any integer. To this end, observe that
\begin{align*}
& \frac{d^kT_1}{d\lambda^k}= -\sum_{\ell_1=0}^1 \widetilde L^\sharp_{\ell_1}(z)\frac{d^k}{d\lambda^k}\left(\mu^{1-\ell_1}
(i\mu^{-1}b\cdot\nabla)^{\ell_1}(P_0 - \lambda^2)^{-1}\mu\right),\\
&\frac{d^kT_2}{d\lambda^k}=\sum_{\ell_1=0}^1\sum_{\ell_2=0}^1\widetilde L^\sharp_{\ell_1}(z)\frac{d^k}{d\lambda^k}\left(\mu^{1-\ell_1}(i\mu^{-1}b\cdot\nabla)^{\ell_1}(P_0-\lambda^2)^{-1}
(\widetilde V\mu^{-1})^{1-\ell_2}(i\nabla\cdot b\mu^{-1})^{\ell_2}\right),\\
& \frac{d^kT_3}{d\lambda^k}=\sum_{\ell_1=0}^1\widetilde L^\sharp_{\ell_1}(z)\frac{d^k}{d\lambda^k}\left(\mu^{1-\ell_1}
(i\mu^{-1}b\cdot\nabla)^{\ell_1}(P_0-\lambda^2)^{-1}\mu\right).
\end{align*}
Therefore, by (\ref{eq:4.7}) and (\ref{eq:2.1}) we have
\begin{equation}\label{eq:4.17}
\left\|\frac{d^kT_j(\lambda,z)}{d\lambda^k}\right\|\lesssim C^{k+1}m_k(|z|+1)^{j-1},\quad \lambda\in\Lambda_\gamma(z),\quad j=1,2,
\end{equation}
\begin{equation}\label{eq:4.18}
\left\|\frac{d^kT_3(\lambda,z)}{d\lambda^k}\right\|\lesssim C^{k+1}m_k,\quad \lambda\in\Lambda_\gamma(z).
\end{equation}
In view of Lemma \ref{3.2}, (\ref{eq:4.9}) and (\ref{eq:4.18}) imply
\begin{equation}\label{eq:4.19}
\left\|\frac{d^k}{d\lambda^k}(I-T_3(\lambda,z))^{-1}\right\|\lesssim C^{k+1}m_k,\quad \lambda\in\Lambda_\gamma(z).
\end{equation}
with possibly a new constant $C>0$. Observe now that
\begin{align*}
\frac{d^kF_1}{d\lambda^k}&= \sum_{\ell_1 = 0}^1 L^\sharp_{0, \ell_1}(z)\frac{d^k}{d\lambda^k}\left(\mu^{1-\ell_1}(-i\mu^{-1}b\cdot\nabla)^{\ell_1}
(P_0-\lambda^2)^{-1}\mu\right) \\
 &-\sum_{\ell_1=0}^1L^\sharp_{0, \ell_1}(z)\frac{d^k}{d\lambda^k}\left(\mu^{1-\ell_1}(-i\mu^{-1}b\cdot\nabla)^{\ell_1}
((P_0-\lambda^2)^{-1}-(P_0-z^2)^{-1})\mu(I-T_3)^{-1}T_1\right),\\
\frac{d^kF_2}{d\lambda^k}&=\sum_{\ell_1=0}^1\sum_{\ell_2=0}^1L^\sharp_{0,\ell_1}(z)\frac{d^k}{d\lambda^k}
\left(\mu^{1-\ell_1}(-i\mu^{-1}b\cdot\nabla)^{\ell_1}(P_0-\lambda^2)^{-1}
(-i\nabla\cdot b\mu^{-1})^{\ell_2}(-\widetilde V\mu^{-1})^{1-\ell_2}\right)\\
&-\sum_{\ell_1=0}^1L^\sharp_{0,\ell_1}(z)\frac{d^k}{d\lambda^k}\left(\mu^{1-\ell_1}(-i\mu^{-1}b\cdot\nabla)^{\ell_1}
((P_0-\lambda^2)^{-1}-(P_0-z^2)^{-1})\mu(I-T_3)^{-1}T_2\right),
\end{align*}
 for any integer $k\ge 1$. Therefore, by Lemma \ref{3.1} and (\ref{eq:4.13}), (\ref{eq:4.17}), (\ref{eq:4.19}), we obtain
 \begin{equation}\label{eq:4.20}
\left\|\frac{d^kF_j(\lambda,z)}{d\lambda^k}\right\|\lesssim C^{k+1}m_k(|z|+1)^{j-2},\quad \lambda\in\Lambda_\gamma(z),\quad j=1,2.
\end{equation}
In view of Lemma \ref{3.2}, (\ref{eq:4.15}) and (\ref{eq:4.20}) with $j=2$ imply
\begin{equation}\label{eq:4.21}
\left\|\frac{d^k}{d\lambda^k}(I-F_2(\lambda,z))^{-1}\right\|\lesssim C^{k+1}m_k,\quad \lambda\in\Lambda_\gamma(z).
\end{equation}
with possibly a new constant $C>0$. By Lemma \ref{3.1}, (\ref{eq:4.16}), (\ref{eq:4.20}) with $j=1$, and (\ref{eq:4.21}), we conclude
 \begin{equation}\label{eq:4.22}
\left\|\frac{d^k}{d\lambda^k}\mu(P-\lambda^2)^{-1}\mu\right\|\lesssim C^{k+1}m_k(|z|+1)^{-1},\quad \lambda\in\Lambda_\gamma(z).
\end{equation}
In particular, (\ref{eq:4.22}) holds with $\lambda=z$. Since $z\in\R$ is arbitrary such that $|z|\ge\delta$, this proves (\ref{eq:1.12})
with $\ell=0$. To prove (\ref{eq:1.12}) with $\ell=1$ we begin from the first identity in (\ref{eq:4.3}). Multiplying on the left by
$\mu\nabla$ and on the right by $\mu$, we get
\begin{equation}\label{eq:4.23}
\begin{split}
\mu\nabla(P-\lambda^2)^{-1}\mu&=\mu\nabla(P_0-\lambda^2)^{-1}\mu\\
&-\mu\nabla(P_0-\lambda^2)^{-1}(\widetilde V\mu^{-1}+i\nabla\cdot b\mu^{-1})\mu(P-\lambda^2)^{-1}\mu\\
&-\mu\nabla(P_0-\lambda^2)^{-1}\mu \mu^{-1}ib\cdot\nabla(P-\lambda^2)^{-1}\mu.
\end{split}
\end{equation}
We now combine (\ref{eq:4.10}) and (\ref{eq:4.23}) to obtain the identity
\begin{equation}\label{eq:4.24}
\begin{split}
\mu\nabla(P-\lambda^2)^{-1}\mu&=\mu\nabla(P_0-\lambda^2)^{-1}\mu\\
&-\mu\nabla(P_0-\lambda^2)^{-1}(\widetilde V\mu^{-1}+i\nabla\cdot b\mu^{-1})\mu(P-\lambda^2)^{-1}\mu\\
&-\mu\nabla(P_0-\lambda^2)^{-1}\mu(I-T_3)^{-1}T_2\mu(P-\lambda^2)^{-1}\mu\\
&-\mu\nabla(P_0-\lambda^2)^{-1}\mu (I-T_3)^{-1}T_1.
\end{split}
\end{equation}
Using (\ref{eq:4.24}) together with Lemma \ref{3.1} and the estimates (\ref{eq:2.1}), (\ref{eq:4.17}), (\ref{eq:4.19}), and (\ref{eq:4.22}), we obtain
\begin{equation}\label{eq:4.25}
\left\|\frac{d^k}{d\lambda^k}\mu\nabla(P-\lambda^2)^{-1}\mu\right\|\lesssim C^{k+1}m_k,\quad \lambda\in\Lambda_\gamma(z),
\end{equation}
which proves (\ref{eq:1.12}) with $\ell=1$.

Finally, we consider the case where the dimension is odd and the condition \eqref{eq:1.11} holds. In this setting, the estimate \eqref{eq:2.1} is valid for all $\lambda \in \R$. Consequently, the identity \eqref{eq:4.5}, together with the bounds \eqref{eq:1.11} and \eqref{eq:4.8}, shows that the operators $\mu \nabla^{\ell}(P-\lambda^2)^{-1}\mu$, $\ell \in \{0, 1\}$, admit continuous extensions not only from $\lambda\in\C^{-}$ to $\R\setminus\{0\}$, but also to $\lambda=0$. It then follows in a straightforward manner that all of the identities and estimates established above remain valid without the auxiliary assumption $|z|\ge\delta$. In particular, the preceding arguments apply uniformly for all $z\in\R$.\\
 \end{proof}

\begin{proof}[Proof of Theorem \ref{1.1} in case b)]
In this case the analysis is similar and even easier than that one in the case a) because of the lack of a magnetic potential.
However, some modifications are necessary due to the presence of an obstacle $\mathcal{O}$. 
In what follows we will sketch the main points. 
Again, we will make use of the resolvent identities obtained in Section 4 of \cite{kn:LSV} and we keep the same notations. 

Let $\eta\in C^\infty(\mathbb{R}^d)$ be of compact support such that
$\eta =1$ on $\mathcal{O}$.  For $\lambda\in\mathbb{C}^-$ we have
 \begin{equation*}
(P_0-\lambda^2)(1-\eta)(P-\lambda^2)^{-1}=([\Delta,\eta]-(1-\eta)V)(P-\lambda^2)^{-1}+1-\eta, \qquad \text{on } L^2(\Omega),
 \end{equation*}
which implies
\begin{equation}\label{eq:4.26}
(1-\eta)(P-\lambda^2)^{-1} =(P_0-\lambda^2)^{-1}([\Delta,\eta]-(1-\eta)V)(P-\lambda^2)^{-1}+(P_0-\lambda^2)^{-1}(1-\eta).
\end{equation}
Let $z\in\mathbb{C}^-$. Similarly,
\begin{equation}\label{eq:4.27}
\begin{split}
(P&-z^2)^{-1}(1-\eta)\\
&=(P-z^2)^{-1}([\eta,\Delta]-(1-\eta)V)(P_0-z^2)^{-1}+(1-\eta)(P_0-z^2)^{-1}, \qquad \text{on } L^2(\R^d).
\end{split}
\end{equation}
 In view of (\ref{eq:4.26}) and (\ref{eq:4.27}),
 \begin{equation*}
 \begin{split}
(P-\lambda^2)^{-1}-(P-z^2)^{-1}&=(\lambda^2-z^2)(P-z^2)^{-1}(P-\lambda^2)^{-1}\\
&=(\lambda^2-z^2)(P-z^2)^{-1}\eta(2-\eta)(P-\lambda^2)^{-1}\\
&+(\lambda^2-z^2)(P-z^2)^{-1}(1-\eta)^2(P-\lambda^2)^{-1}\\
&=(\lambda^2-z^2)(P-z^2)^{-1}\eta(2-\eta)(P-\lambda^2)^{-1}\\
&+(1-\eta+(P-z^2)^{-1}( [\eta,\Delta]-(1-\eta)V))((P_0-\lambda^2)^{-1}\\
&-(P_0-z^2)^{-1})(1-\eta+([\Delta,\eta]-(1-\eta)V)(P-\lambda^2)^{-1}).
\end{split}
 \end{equation*}
Multiplying both sides of this identity by $\mu$ we get
\begin{equation}\label{eq:4.28}
\begin{split}
\mu(P-\lambda^2)^{-1}\mu-\mu(P-z^2)^{-1}\mu&=(\lambda^2-z^2)\mu(P-z^2)^{-1}\eta(2-\eta)(P-\lambda^2)^{-1}\mu\\
 &+Q_1(z)(\mu(P_0-\lambda^2)^{-1}\mu-\mu(P_0-z^2)^{-1}\mu)Q_2(\lambda),
 \end{split}
 \end{equation}
 where 
 \begin{align*}
& Q_1(z)=1-\eta+\mu(P-z^2)^{-1}([\eta,\Delta]-(1-\eta)V)\mu^{-1},\\
& Q_2(\lambda)=1-\eta+\mu^{-1}([\Delta,\eta]-(1-\eta)V)(P-\lambda^2)^{-1}\mu.
 \end{align*}
 We rewrite (\ref{eq:4.28}) in the form
 \begin{equation}\label{eq:4.29}
 \begin{split}
 &(I-K(\lambda,z))\mu(P-\lambda^2)^{-1}\mu=\mu(P-z^2)^{-1}\mu\\
 &+Q_1(z)(\mu(P_0-\lambda^2)^{-1}\mu-\mu(P_0-z^2)^{-1}\mu)(1-\eta),
 \end{split}
 \end{equation}
 where 
 \begin{equation*}
 \begin{split}
 K(\lambda,z)&=(\lambda^2-z^2)\mu(P-z^2)^{-1}\eta(2-\eta)\mu^{-1}\\
 &+Q_1(z)(\mu(P_0-\lambda^2)^{-1}-\mu(P_0-z^2)^{-1})([\Delta,\eta]-(1-\eta)V)\mu^{-1}.
 \end{split}
 \end{equation*}
 As in the previous case, in view of the limiting absorption principle, the identities (\ref{eq:4.28}) and (\ref{eq:4.29}) extend to all $\lambda,z\in\R\setminus \{0\}$. We fix $z\in\R$, $|z|\ge\delta$, $0<\delta\ll 1$ being arbitrary, and consider the above operators as functions of $\lambda\in\Lambda_\gamma(z):=\{\lambda\in\R:|\lambda-z|\le\gamma\}$, where $0<\gamma<\delta/2$ is a parameter
independent of $z$ to be fixed below.
It follows from (\ref{eq:4.2}) that
\begin{equation}\label{eq:4.30}
\|Q_1(z)\|\lesssim 1.
\end{equation}
By (\ref{eq:2.1}) with $k=1$, (\ref{eq:4.2}), (\ref{eq:4.30}), as above, we get
\begin{equation}\label{eq:4.31}
\|K(\lambda,z)\|\lesssim \gamma\le 1/2
\end{equation}
for $\lambda\in\Lambda_\gamma(z)$ with a small constant $\gamma>0$. Observe that
\begin{equation*}
 \begin{split}
\frac{d^kK}{d\lambda^k}&=\mu(P-z^2)^{-1}\eta(2-\eta)\mu^{-1}\frac{d^k\lambda^2}{d\lambda^k}\\
 &+Q_1(z)\frac{d^k}{d\lambda^k}\mu(P_0-\lambda^2)^{-1}([\Delta,\eta]-(1-\eta)V)\mu^{-1}
 \end{split}
 \end{equation*}
 for any integer $k\ge 1$. Therefore, by (\ref{eq:2.1}) together with (\ref{eq:4.30}) and (\ref{eq:4.31}), we get
\begin{equation}\label{eq:4.32}
\left\|\frac{d^kK(\lambda,z)}{d\lambda^k}\right\|\lesssim C^{k+1}m_k,\quad \lambda\in\Lambda_\gamma(z).
\end{equation}
In view of Lemma \ref{3.2}, (\ref{eq:4.31}) and (\ref{eq:4.32}) imply
\begin{equation}\label{eq:4.33}
\left\|\frac{d^k}{d\lambda^k}(I-K(\lambda,z))^{-1}\right\|\lesssim C^{k+1}m_k,\quad \lambda\in\Lambda_\gamma(z).
\end{equation}
with possibly a new constant $C>0$. Furthermore, by (\ref{eq:4.29}) we have
\begin{equation*}
 \frac{d^k}{d\lambda^k}\left((I-K(\lambda,z))\mu(P-\lambda^2)^{-1}\mu\right)=
 Q_1(z)\frac{d^k}{d\lambda^k}\mu(P_0-\lambda^2)^{-1}\mu(1-\eta)
 \end{equation*}
 for any integer $k\ge 1$. Therefore, (\ref{eq:4.30}) and (\ref{eq:2.1}) with $\ell =0$ imply
 \begin{equation}\label{eq:4.34}
\left\|\frac{d^k}{d\lambda^k}\left((I-K(\lambda,z))\mu(P-\lambda^2)^{-1}\mu\right)
\right\|\lesssim C^{k+1}m_k(|z|+1)^{-1},\quad \lambda\in\Lambda_\gamma(z).
\end{equation}
Using Lemma \ref{3.1} we deduce from (\ref{eq:4.33}) and (\ref{eq:4.34}) that the estimate (\ref{eq:4.22}) still holds in case b),
which implies (\ref{eq:1.12}) with $\ell=0$. The estimate (\ref{eq:1.12}) with $\ell=1$ can be deduced from (\ref{eq:4.22})
in the same way as in the case a), using the identity
\begin{equation}\label{eq:4.35}
\begin{split}
\mu\nabla(P-\lambda^2)^{-1}\mu-\mu\nabla(P-z^2)^{-1}\mu&=(\lambda^2-z^2)\mu\nabla(P-z^2)^{-1}\eta(2-\eta)(P-\lambda^2)^{-1}\mu\\
 &+\widetilde Q_1(z)(\mu(P_0-\lambda^2)^{-1}\mu-\mu(P_0-z^2)^{-1}\mu)Q_2(\lambda),
 \end{split}
 \end{equation}
 where
 \begin{equation*}
\widetilde Q_1(z)=\mu\nabla(1-\eta)\mu^{-1}+\mu\nabla(P-z^2)^{-1}([\eta,\Delta]-(1-\eta)V)\mu^{-1}.
\end{equation*}

To address the situation where the dimension $d \ge 3$ is odd and \eqref{eq:1.11} holds, we argue as we did above for the case b). Since, in odd dimensions, \eqref{eq:2.1} is valid for all $\lambda \in \R$, \eqref{eq:4.28} and \eqref{eq:4.35}, together with \eqref{eq:1.11} and \eqref{eq:4.8}, show that $\mu \nabla^{\ell}(P-\lambda^2)^{-1}\mu$, $\ell \in \{0, 1\}$, admit continuous extensions to $\R$. Proceeding from \eqref{eq:4.29}, we may then verify that all subsequent identities and estimates established above remain valid uniformly for $z\in\R$.\\
\end{proof}

\section{Time decay estimates} \label{time decay section}

In this section we sketch how the arguments from \cite[Section 6]{kn:LSV} adapt to yield Theorem \ref{1.2}. Start, for some choice of $\Theta$, with the function $\psi(\lambda)$ defined in Section \ref{introduction section}. For $\delta > 0$ put $\psi_\delta(\lambda) \defeq \psi(\lambda/\delta)$. We have $0 \le \psi_\delta(\lambda) \le 1$, $\psi_\delta = 0$  near $(-\infty, \delta]$, and $\psi_\delta = 1$ near $[2\delta, \infty)$; for each $k\geq 0$,
\begin{equation*}
\left|\frac{d^k}{d\lambda^k}\psi_\delta(\lambda)\right|\leq (C\delta^{-1})^km_k.
\end{equation*}

An adaptation of the proof of \cite[Prop. 6.2]{kn:LSV} gives the following

\begin{prop} \label{integral estimates}
For all integers $k \ge 0$ and all $t>1$,
\begin{equation}\label{integral est L2}
\begin{split}
&\int_t^\infty \left\|\mu\cos(t'\sqrt{P})\psi_\delta(P^{1/2})\mu f\right\|_{L^2}^2dt'+
\int_t^\infty \left\|\mu\nabla^{\ell}P^{-1/2}\sin(t'\sqrt{P})\psi_\delta(P^{1/2})\mu f\right\|_{L^2}^2dt'\\
&\le  C^{2k+2}m_k^2t^{-2k}\|f\|_{L^2}^2,\qquad  f\in L^2,
\end{split}
\end{equation}
\begin{equation}\label{integral est H1}
\begin{split}
&\int_t^\infty \left\|\mu P^{1/2}\sin(t'\sqrt{P})\psi_\delta(P^{1/2})\mu f\right\|_{L^2}^2dt'+
\int_t^\infty \left\|\mu\nabla^{\ell}\cos(t'\sqrt{P})\psi_\delta(P^{1/2})\mu f\right\|_{L^2}^2dt'\\
&\le  C^{2k+2}m_k^2t^{-2k}\|f\|_{H^1}^2,\qquad  f\in H^1,
\end{split}
\end{equation}
where $\ell\in\{0,1\}$, and $C>0$ is a constant independent of $k$, $t$ and $f$. 
 If the dimension $d$ is odd and the condition (\ref{eq:1.11})
 is assumed, then the estimates \eqref{integral est L2} and \eqref{integral est H1} hold with $\psi\equiv 1$ for all integers $k\ge 0$. 
\end{prop}

\begin{proof}[Proof of Theorem \ref{1.2}]

The proof of Theorem \ref{1.1} now follows by the same argument used to prove \cite[Theorem 1.1]{kn:LSV}. In particular, set
\begin{equation*}
u(t) = \sin(t \sqrt{P}) P^{-1/2} \psi_\delta(P^{1/2}) \mu f,\quad f \in D(P).
\end{equation*}
Computing the time derivative of the quantity 
\begin{equation*}
\left\|\mu \partial_tu(t)\right\|^2_{L^2}+\left\|\mu(i\nabla+b)u(t)\right\|^2_{L^2}
+\left\|\mu u(t)\right\|^2_{L^2}, 
\end{equation*}
we get
\begin{equation}\label{time deriv of energy}
\begin{split}
&\frac{d}{dt}\left(\left\|\mu \partial_tu(t)\right\|^2_{L^2}+\left\|\mu(i\nabla+b)u(t)\right\|^2_{L^2}
+\left\|\mu u(t)\right\|^2_{L^2}\right)\\
&=2{\rm \Re}\langle \mathcal{N}(\mu)u(t),\mu \partial_tu(t)\rangle_{L^2} + 2{\rm Re}\langle\mu\partial_t u(t),\mu u(t)\rangle_{L^2}\\
&\le 2\left\|\mu \partial_tu(t)\right\|^2_{L^2}+\left\|\mu u(t)\right\|^2_{L^2}+\left\|\mathcal{N}(\mu)u(t)\right\|^2_{L^2},
\end{split}
\end{equation}
where
\begin{equation*}
\mathcal{N}(\mu)=\mu^{-1}\left([-\Delta, \mu^2] + 2ib \cdot \nabla \mu^2 - V \mu^2 - (i\nabla+b)\cdot[i\nabla,\mu^2]\right) =
\sum_{\ell=0}^1O_{\ell}(\mu)\nabla^{\ell}.
\end{equation*}
By \eqref{time deriv of energy}, for all $T>t>1$, 
\begin{equation} \label{apply FTC}
\begin{split}
&\left\|\mu \partial_tu(t)\right\|_{L^2}^2+\left\|\mu(i\nabla+b)u(t)\right\|_{L^2}^2+\left\|\mu u(t)\right\|_{L^2}^2\\
&\lesssim \left\|\mu \partial_tu(T)\right\|_{L^2}^2+\left\|\mu(i\nabla+b)u(T)\right\|_{L^2}^2+\left\|\mu u(T)\right\|_{L^2}^2\\
&+\int_t^T\left\|\mu\partial_t u(t')\right\|_{L^2}^2dt'+\sum_{\ell=0}^1\int_t^T\left\|\mu\nabla^{\ell}u(t')\right\|_{L^2}^2dt'.
\end{split}
\end{equation}
On the other hand, it follows from \eqref{integral est L2} with $k=0$ that there exists a sequence $T_j\to\infty$ such that
\begin{equation}\label{lim to zero}
\lim_{T_j\to\infty} \left(\left\|\mu \partial_tu(T_j)\right\|_{L^2}^2+\left\|\mu(i\nabla+b)u(T_j)\right\|_{L^2}^2
+\left\|\mu u(T_j)\right\|_{L^2}^2\right)=0.
\end{equation}
Therefore, using \eqref{apply FTC} with $T=T_j$ and taking the limit as $T_j\to\infty$, in view of \eqref{lim to zero}, we obtain
\begin{equation}\label{local energy bd}
\begin{split}
&\left\|\mu \partial_tu(t)\right\|_{L^2}^2+\left\|\mu(i\nabla+b)u(t)\right\|_{L^2}^2+\left\|\mu u(t)\right\|_{L^2}^2\\
&\lesssim\int_t^\infty\left\|\mu\partial_t u(t')\right\|_{L^2}^2dt'+\sum_{\ell=0}^1
\int_t^\infty\left\|\mu\nabla^{\ell}u(t')\right\|_{L^2}^2dt'.
\end{split}
\end{equation}
By \eqref{integral est L2} and \eqref{local energy bd},
\begin{equation} \label{decay before fix k}
\left\|\mu\partial_t u(t)\right\|_{L^2}+\left\|\mu(i\nabla+b)u(t)\right\|_{L^2}+\left\|\mu u(t)\right\|_{L^2}\le C^{k+1}m_kt^{-k}\|f\|_{L^2},
\end{equation}
Now take $k = k((Ce)^{-1}t)$ to be the biggest integer such that $m_k^{1/k} \le (Ce)^{-1}t$, which implies $C^{k+1} m_k t^{-k} \le Ce^{-k((Ce)^{-1}t)}$, confirming \eqref{eq:1.13} for $f \in D(P)$. Then (\ref{eq:1.13}) follows for general $f \in L^2(\Omega)$ since $D(P)$ is dense in $L^2(\Omega)$.

To get (\ref{eq:1.14}), we apply the above analysis
to
\begin{equation*}
u(t) = \cos(t \sqrt{P}) \psi_\delta(P^{1/2})\mu f,\qquad f \in D(P),
\end{equation*}
 and use the estimate \eqref{integral est H1} instead of \eqref{integral est L2} to conclude that  \eqref{decay before fix k} holds with $\|f\|_{L^2}$ in the right-hand side replaced by $\|f\|_{H^1}$. Then \eqref{eq:1.14} follows from \eqref{decay before fix k} and the fact that $D(P)$ is dense in $D(P^{1/2})$ with respect to the norm $f \mapsto (\|f\|^2_{L^2} + \|P^{1/2} f\|^2_{L^2})^{1/2}$.

In odd dimensions, under the condition (\ref{eq:1.11}),
the preceding computations hold with $\psi_\delta \equiv 1$ because so do the estimates \eqref{integral est L2} and \eqref{integral est H1}. \\
\end{proof}

\appendix

\section{Construction of a Gevrey cutoff function} \label{gevrey cutoff appendix}

In this appendix, we outline the construction of the Gevrey-class cutoff function $\zeta$ utilized in the statement of Theorem \ref{1.2}. In particular, for each $0 < s < 1$, we construct $\zeta \in C^{\infty}(\mathbb{R}; [0,1])$ with $\zeta(\sigma) = 0$ for $\sigma \le 1/2$ and $\sigma \ge 2$, and such that $|\partial_\sigma^k \zeta(\sigma)| \le C^{k+1}(k!)^{1/s}$ and $\int_{\R} \zeta(\sigma) d\sigma = 1$.

The construction builds upon the function $f : \R \to \R$ given by
\begin{equation*}
    f(x) = \begin{cases} 
    e^{-x^{-a}} & x > 0, \\
    0 & x \le 0, 
    \end{cases}
\end{equation*}
with $a > 0$. It is easy to verify that $f \in C^\infty(\mathbb{R})$ by induction on the derivatives. To establish the Gevrey bounds, we utilize a complex variable argument. 

We extend $f$ to the complex plane by taking the principal branch of the logarithm, $z^{-a} = e^{-a(\ln|z| + i \arg z)}$. First, consider the case $a > 1$. For fixed $x > 0$, we consider the contour integral of $f(z)$ along a circle $\gamma_a$ centered at $x$ with radius $R = x/a$. By the Cauchy integral formula:
\begin{equation*}
    f^{(k)}(x) = \frac{k!}{2\pi i} \oint_{\gamma_a} \frac{e^{-z^{-a}}}{(z-x)^{k+1}} \, dz.
\end{equation*}
Estimating yields:
\begin{equation*}
    |f^{(k)}(x)| \le \frac{k!}{R^k} \sup_{z \in \gamma_a} \left| e^{-z^{-a}} \right| = k! \left(\frac{a}{x}\right)^k e^{-\inf_{z \in \gamma_a} \Re(z^{-a})}.
\end{equation*}

Parameterizing $\gamma_a$ as $z = x(1 + a^{-1}e^{i\theta})$, $\theta \in [0,2 \pi]$, it is straightforward to show that the argument of $1 + a^{-1} e^{i\theta}$ is bounded in magnitude by $\arctan(1/\sqrt{a^2 - 1})$. Consequently, the magnitude of the argument of $z^{-a}$ is bounded by $a \arctan(1/\sqrt{a^2 - 1})$, which is strictly less than $\pi/2$ for any $a > 1$. Whence
\begin{equation*}
\inf_{z \in \gamma_a} \Re(z^{-a}) \ge x^{-a} \left(1 - \frac{1}{a}\right)^{-a} \cos\left(a \arctan\frac{1}{\sqrt{a^2-1}}\right) \qefed c_a x^{-a}.
\end{equation*}
Substituting this back into the Cauchy estimate and applying Stirling's approximation, we obtain,
\begin{equation} \label{gevrey bd f}
    |f^{(k)}(x)| \le k! \left(\frac{a}{x}\right)^k e^{-c_a x^{-a}} \le a^k k! \sup_{y > 0} y^{1/a} e^{-c_a y} = a^k k!  \left( \frac{k}{ac_a} \right)^{k/a} e^{-k/a} \le C^k (k!)^{1 + \frac{1}{a}}
\end{equation}
for some $C > 0$ independent of $k$. This is the desired Gevrey bound for order $s = 1 + a^{-1}$ and $a > 1$

The case $0 < a \le 1$ is simpler because, in the above argument, one may choose the contour radius to be $R = x/2$, independent of $a$. Thus it is easier to bound $\Re(z^{-a})$ from below on the contour. 

With \eqref{gevrey bd f} established, we proceed with the construction of $\zeta$ by defining 
\begin{equation}
    \rho(x) = \frac{f(x)}{f(x) + f(1-x)} \in [0,1].
\end{equation}
It is clear that $\rho(x) = 0$ for $x \le 0$ and $\rho(x) = 1$ for $x \ge 1$. Moreover, since the denominator is strictly positive, $\rho$ has the the same Gevrey regularity as $f$ (the proofs of Lemmas \ref{3.1} and \ref{3.2} hold also for real valued functions). 

Next, set
\begin{equation}
    \chi(\sigma) = \rho\left(4\left(\sigma - \frac{1}{2}\right)\right) \rho\left(4\left(2 - \sigma\right)\right).
\end{equation}
It is easy to check that $\chi(\sigma) = 0$ for $\sigma \notin [1/2, 2]$, and $\chi(\sigma) = 1$ for $\sigma \in [3/4, 7/4]$. Consequently, $I \defeq \int_{\mathbb{R}} \chi(\sigma) \, d\sigma > 1$. Finally, defining $\zeta(\sigma) = I^{-1} \chi(\sigma)$ yields the desired function appearing in the statement of Theorem \ref{1.2}.

\section{The limiting absorption principle} \label{limiting absorption appendix}

In this appendix we address the limiting absorption principle for the magnetic Schr\"odinger operator 
\begin{equation*}
P=(i\nabla+b(x))^2+V(x) : L^2(\R^d) \to L^2(\R^d), \qquad d \ge 2.
\end{equation*}
where the electric potential $V\in L^\infty(\mathbb{R}^d,\mathbb{R})$ and the magnetic potential  $b\in L^\infty(\mathbb{R}^d,\mathbb{R}^d)$ satisfy
\begin{equation*}
|V(x)|+|b(x)|\le C\langle x\rangle^{-\rho}, \qquad C>0,\,\rho>1.
\end{equation*}
The limiting absorption principle is well known to hold for this class of short range potentials. Its relevance to our work motivates us to provide a concise proof here. In particular we show that for $s > 1/2$, the map $\lambda \mapsto (P-\lambda^2)^{-1} : \langle x\rangle^{-s}H^{-1}(\mathbb{R}^d)\to \langle x\rangle^{s}H^1(\mathbb{R}^d)$ admits a continuous extension, in the operator norm topology, from $\Im \lambda < 0$ to $\R \setminus \{0\}$. This is equivalent to establishing a continuous extension for $\langle x \rangle^{-s} \partial^\alpha_x(P-\lambda^2)^{-1} \partial^\beta_x \langle x \rangle^{-s} : L^2(\R^d) \to L^2(\R^d)$, for all multi-indices $\alpha$ and $\beta$ such that $|\alpha|, \, |\beta| \le 1$. Once the extension from $\Im \lambda < 0$ to $\R \setminus \{0\}$ is established, the corresponding extension from $\Im \lambda > 0$ to $\R \setminus \{0\}$ follows because the $L^2(\R^d) \to L^2(\R^d)$ norm of $\langle x \rangle^{-s} \partial^\alpha_x(P-\lambda^2)^{-1} \partial^\beta_x \langle x \rangle^{-s}$ and its adjoint $(-1)^{|\alpha| + |\beta|}\langle x \rangle^{-s} \partial^\beta_x(P-\overline{\lambda}^2)^{-1} \partial^\alpha_x \langle x \rangle^{-s}$ coincide.

The proof utilizes the resolvent bound \eqref{eq:4.2} and the resolvent identity \eqref{eq:4.3}. We will also need (and prove) the following estimate for the free resolvent squared: for any $s > 1/2$ and $0 < \delta < M$, there exist $C, \varepsilon  > 0$ so that
\begin{equation} \label{bd resolv square}
\begin{gathered}
\| \langle x \rangle^{-s} \partial^\alpha_x (P_0 - \lambda^2)^{-2} \partial^\beta_x  \langle x \rangle^{-s} \|_{L^2 \to L^2} \le C, \\  \delta \le |\Re \lambda | \le M, \, 0 < |\Im \lambda| \le \varepsilon, \, |\alpha| + |\beta| \le 2.
\end{gathered}
\end{equation}
 
Our goal is to show that, for any $0 < \delta < M$, there exists $\varepsilon > 0$ so that the $\lambda$-derivative of $\langle x \rangle^{-s}  \partial^\alpha_x(P-\lambda^2)^{-1}  \partial^\beta_x \langle x \rangle^{-s}$ is uniformly bounded for $\delta \le |\Re \lambda| \le M$ and $- \varepsilon \le \Im \lambda < 0$. Establishing this bound implies the uniform continuity of the weighted resolvent in these strips, yielding the desired continuous extension from $\Im \lambda < 0$ to $\R \setminus \{0\}$. Observe that 
\begin{equation*}
 \frac{d}{d\lambda}   \langle x \rangle^{-s} \partial^\alpha_x (P - \lambda^2)^{-1}  \partial^\beta_x \langle x \rangle^{-s} = 2\lambda \langle x \rangle^{-s}  \partial^\alpha_x (P - \lambda^2)^{-2}  \partial^\beta_x \langle x \rangle^{-s}. \end{equation*}
Thus it suffices to bound the $L^2(\R^d) \to L^2(\R^d)$ norm of the right side.

Note that for $u$ in the domain $D(P)$ of the operator $P$, we have in the distributional sense,
\begin{equation}
\begin{split}
Pu &=P_0 u+i\nabla\cdot (bu)+ib\cdot\nabla u+\widetilde Vu\\
&=P_0u+Qu,
\end{split}
\end{equation}
where $P_0 = -\Delta$ denotes the free Laplacian, $\widetilde V=V+|b|^2$, and $Q=\widetilde V + i\nabla\cdot (bu)+ib\cdot\nabla u$. Consequently, for $\Im \lambda < 0$,
\begin{equation} \label{first step square resolv id}
\begin{split}
 -\lambda^2\langle x \rangle^{-s}  \partial^\alpha_x(&P-\lambda^2)^{-2}  \partial^\beta_x\langle x \rangle^{-s}\\
 &=\langle x \rangle^{-s}  \partial^\alpha_x(P-\lambda^2)^{-1}  \partial^\beta_x \langle x \rangle^{-s}- \langle x \rangle^{-s}(P-\lambda^2)^{-1}P(P-\lambda^2)^{-1}\langle x \rangle^{-s},\\
 &=\langle x \rangle^{-s}  \partial^\alpha_x(P-\lambda^2)^{-1}  \partial^\beta_x\langle x \rangle^{-s}-\langle x \rangle^{-s}  \partial^\alpha_x (P-\lambda^2)^{-1}Q(P-\lambda^2)^{-1}  \partial^\beta_x \langle x \rangle^{-s}\\
 &-\langle x \rangle^{-s} \partial^\alpha_x(P-\lambda^2)^{-1}P_0(P-\lambda^2)^{-1}  \partial^\beta_x\langle x \rangle^{-s}.
 \end{split}
 \end{equation}
Inserting $1 = \langle x \rangle^{-\frac{\rho}{2}} \langle x \rangle^{\frac{\rho}{2}}$ at appropriate places within the operator $Q$ yields.
\begin{equation} \label{Q supplies weights}
Q = \langle x \rangle^{-\frac{\rho}{2}} (\langle x \rangle^{\rho} \widetilde V) \langle x \rangle^{-\frac{\rho}{2}} + i \nabla \cdot \langle x \rangle^{-\frac{\rho}{2}} (\langle x \rangle^{\rho} b) \langle x \rangle^{-\frac{\rho}{2}} + i \langle x \rangle^{-\frac{\rho}{2}} (\langle x \rangle^{\rho}b)\cdot \langle x \rangle^{-\frac{\rho}{2}} \nabla.  
\end{equation}
Combined with \eqref{eq:4.2}, \eqref{Q supplies weights} shows that for fixed $0 < \delta < M$ , there exists $\varepsilon > 0$ so that if $\delta < |\Re \lambda| \le M$ and $ \Im \lambda < 0$, the third line of \eqref{first step square resolv id} is uniformly bounded $L^2(\R^d) \to L^2(\R^d)$. 

We expand the fourth line of \eqref{first step square resolv id} using \eqref{eq:4.3}. Specifically,
\begin{equation} \label{unpack P0 in middle term}
\begin{split}
\langle x \rangle^{-s} \partial^\alpha_x(&P-\lambda^2)^{-1}P_0(P-\lambda^2)^{-1} \partial^\beta_x\langle x \rangle^{-s} \\
&= \langle x \rangle^{-s}  \partial^\alpha_x \big[ (P_0 - \lambda^2)^{-1} - (P- \lambda^2)^{-1} Q (P_0 - \lambda^2)^{-1} \big] P_0 \\
&\quad \cdot \big[ (P_0 - \lambda^2)^{-1} - (P_0- \lambda^2)^{-1} Q (P - \lambda^2)^{-1} \big]  \partial^\beta_x \langle x \rangle^{-s}.
\end{split}
\end{equation}
We rewrite the central factor as $P_0 = (P_0 - \lambda^2) + \lambda^2$. Applying this and \eqref{Q supplies weights} shows that \eqref{unpack P0 in middle term} can be written as a sum of operators, each of which is a product of bounded operators with polynomial dependence on $
\lambda$, and factors  of the form
\begin{equation*}
\langle x \rangle^{-s} \partial^\mu_x(P - \lambda^2)^{-1}\partial^\nu_x \langle x \rangle^{-s}, \quad
\langle x \rangle^{-s} \partial^\mu_x(P_0 - \lambda^2)^{-1}\partial^\nu_x \langle x \rangle^{-s}, \quad \text{or} \quad
\langle x \rangle^{-s} \partial^\mu_x(P_0 - \lambda^2)^{-2}\partial^\nu_x \langle x \rangle^{-s},
\end{equation*}
with multi-indices $\mu$ and $\nu$ satisying $|\mu| + |\nu| \le 1$. Therefore, for any fixed $0 < \delta < M$, there exists $\varepsilon > 0$ such that the factors involving single resolvents (of $P$ or $P_0$) are controlled by \eqref{eq:4.2}, while those involving the squared free resolvent are controlled by \eqref{bd resolv square}.

It thus remains to show \eqref{bd resolv square}. The starting point is the following resolvent identity established in \cite[Appendix E]{kn:llst25} for $s > 3/2$:
\begin{equation} \label{identity to insert Laplacian}
\begin{split}
-\lambda^2 \langle x \rangle^{-s} (&P_0 - \lambda^2)^{-2} \langle x \rangle^{-s} \\
&= - \tfrac{1}{2} \langle x \rangle^{-s} (P_0 - \lambda^2)^{-1}(2P_0) (P_0 - \lambda^2)^{-1} \langle x \rangle^{-s}  \\
& +\langle x \rangle^{-s} (P_0 - \lambda^2)^{-1} \langle x \rangle^{-s}\\
&=  \frac{1}{2} \langle x \rangle^{-s} (P_0 - \lambda^2)^{-1} (\nabla \cdot x  + (d-1))\langle x \rangle^{-s}  + \frac{1}{2} \langle x \rangle^{-s} (P_0 - \lambda^2)^{-1}  \langle x \rangle^{-s}\\
& -\frac{1}{2}  \langle x \rangle^{-s} x \cdot  \nabla (P_0 - \lambda^2)^{-1}  \langle x \rangle^{-s}. 
\end{split} 
\end{equation}
Using \eqref{eq:4.2} in \eqref{identity to insert Laplacian} yields \eqref{bd resolv square} subject to the restrictions $|\alpha| = |\beta| = 0$ and $s > 3/2$. Standard elliptic regularity arguments (and the fact that $\partial^\alpha_x$ commutes with $(P_0 - \lambda^2)^{-1}$) allow us to extend this bound to the general case $|\alpha| + |\beta| \le 1$, provided $s > 3/2$.

To relax the weight condition to $s > 1/2$, we employ a bootstrapping argument. First, assuming $s > 5/2$, we differentiate \eqref{identity to insert Laplacian} with respect to $\lambda$. The resulting expression shows $\langle x \rangle^{-s} (P - \lambda^2)^{-3} \langle x \rangle^{-s}$ can be controlled using the $s > 3/2$ bound we just established for the squared resolvent. This proves that if $s > 5/2$, the map $\lambda \mapsto \langle x \rangle^{-s} (P_0 - \lambda^2)^{-2} \langle x \rangle^{-s}$ is Lipschitz continuous in the region $\delta \le |\Re \lambda| \le M, \,  |\Im \lambda | \le \varepsilon$.

We relax the weight requirement as in \cite[Section 3]{kn:cavo04}, proceeding in steps. First, we extend continuity to the range $s > 3/2$. Let $\mathbf{1}_{\le A}$ and $\mathbf{1}_{\ge A}$ denote the characteristic functions of the sets $\{ |x| \le A\}$ and $\{ |x| \ge A \}$, respectively. Fix $s > 3/2$ and choose $\sigma > 5/2$. Let $\lambda_1, \lambda_2 \in \{z \in \C : \delta \le \Re z \le M, \, \epsilon \le \Im \lambda < 0\}$ (the case of negative real parts is identical). Then
\begin{equation*}
\begin{split}
\| \langle x \rangle^{-s} &((P_0- \lambda_1^2)^{-2} - (P_0 - \lambda^2_2)^{-2}) \langle x \rangle^{-\sigma} \|_{L^2 \to L^2} \\
&\le \langle A \rangle^{2(\sigma-s)} \| \mathbf{1}_{\le A} \langle x \rangle^{-\sigma} ((P_0- \lambda^2_1)^{-2} - (P_0 - \lambda^2_2)^{-2}) \langle x \rangle^{-\sigma} \|_{L^2 \to L^2} \\
&+  \sum_{j = 1}^2 \| \mathbf{1}_{\ge A} \langle x \rangle^{-s} (P_0- \lambda^2_1)^{-2} \langle x \rangle^{-\sigma} \|_{L^2 \to L^2} \end{split}
\end{equation*}
The first term is bounded by $C A^{2(\sigma-s)} |\lambda_1 - \lambda_2|$ using the Lipschitz continuity since $\sigma > 5/2$. For the other terms,  fix $\epsilon > 0$ such that $s - \epsilon > 3/2$. Then
\begin{equation*}
\sum_{j = 1}^2 \| \mathbf{1}_{\ge A} \langle x \rangle^{-s} (P_0- \lambda_j^2)^{-2} \langle x \rangle^{-\sigma} \| \le A^{-(s -\frac{3}{2} - \epsilon)}  \sum_{j = 1}^2 \| \langle x \rangle^{-\frac{3}{2}-\epsilon} (P_0- \lambda_j^2)^{-2} \langle x \rangle^{-\sigma} \|.
\end{equation*}
We previously concluded that the norms on the right are uniformly bounded. Choosing $A = |\lambda_1 - \lambda_2|^{-\kappa}$ for $\kappa > 0$ sufficiently small yields H\"older continuity for $\lambda \mapsto \langle x \rangle^{-s} (P_0 - \lambda^2)^{-2} \langle x \rangle^{-\sigma}$ with $s > 3/2$ and $\sigma > 5/2$.

We repeat this argument to replace the weight on the right by $\langle x \rangle^{-s}$ with $s > 3/2$, and then iterate twice more to reduce the weights on both sides to $\langle x \rangle^{-s}$ with $s > 1/2$. The H\"older exponents used in this process may decrease at each step. Finally, applying elliptic regularity extends \eqref{bd resolv square} to the case $|\alpha| + |\beta| \le 2$.

\section{Kernel of the operator $\cos(t \sqrt{P_0})$} \label{cosine kernel appendix}

Let $P_0$ denote the self-adjoint realization of $-\Delta$ on $L^2(\R^d)$, $d \ge 2$. In this section we review well known properties of the kernel of the operator $\cos(t \sqrt{P_0}) : L^2(\R^d) \to L^2(\R^d)$ which we utilize in Section \ref{free resolv est section}. This operator also takes the form 
\begin{equation*}
\begin{split}
\cos(t \sqrt{P_0}) &=  \mathcal{F}^{-1} \cos(t |\xi|) \mathcal{F} \\
&= \tfrac{1}{2} \mathcal{F}^{-1} \big( e^{it|\xi|} + e^{-it|\xi|}\big) \mathcal{F},
\end{split}
\end{equation*}
where $\mathcal{F}$ denotes Fourier transform. It is well known \cite[Chapter III, Section 2]{kn:S70} that for $\Re w > 0$, 
\begin{equation*}
(\mathcal{F}^{-1} e^{-w |\xi | } \mathcal{F} u)(x) = C_d w \int_{\R^d} \frac{u(y)}{(\sqrt{w^2 + |x - y|^2})^{d+1}} dy,  \qquad u \in L^2(\R^d),
\end{equation*} 
for some $C_d  > 0$ depending on $d$, and where the branch of the square root is chosen so that $\Im \sqrt{z} > 0$ if $\Im z \neq 0$.  Therefore, setting $w = t \pm i \varepsilon$ for $\varepsilon > 0$,
\begin{equation} \label{compute kernel}
(\mathcal{F}^{-1} e^{ \pm i |\xi |t} \mathcal{F} u)(x) = \lim_{\varepsilon \to 0} \mp C_d i (t \pm i \varepsilon) \int_{\R^d} \frac{u(y)}{(\sqrt{-(t \pm i \varepsilon)^2 + |x - y|^2})^{d+1}} dy,
\end{equation}
where the limit is in the sense of convergence in $L^2(\R^d)$. As in Section \ref{free resolv est section}, let $\phi \in C^\infty_0(\R^d ; [0, \infty])$ with $\phi(x) = 1$ for $|x| \le 1/8$, $\phi(x) = 0$ for $|x| \ge 1/4$. Whence $x\in {\rm supp}\,\phi(x/t)$ and
$y\in {\rm supp}\,\phi(y/t)$ imply $|x-y|\le t/2$. Thus, applying the $\phi(\cdot/t)$ to both sides of $\cos(t \sqrt{P_0})$ gives
\begin{equation*}
\begin{split}
\phi(\cdot /t) \cos(t \sqrt{P_0}) \phi(\cdot/ t) u &= \lim_{\varepsilon \to 0} \Big[ -C_d i(t + i \varepsilon) \int_{\R^d} \phi(x/t) \phi(y/t)\frac{u(y)}{(\sqrt{-(t + i \varepsilon)^2 + |x - y|^2})^{d+1}} dy \\
&+  C_d i(t - i \varepsilon) \int_{\R^d} \phi(x/t) \phi(y/t)\frac{u(y)}{(\sqrt{-(t - i \varepsilon)^2 + |x - y|^2})^{d+1}} dy \Big] \\
&= \lim_{\varepsilon \to 0} C_d \int_{\R^d}  \phi(x/t) \phi(y/t) \Re \Big[ \frac{-i(t + i\varepsilon)}{(\sqrt{-(t + i \varepsilon)^2 + |x - y|^2})^{d+1}} \Big] u(y) dy.
\end{split}
\end{equation*}
The resulting kernels depending on $\varepsilon$ converge in Hilbert-Schmidt norm as $\varepsilon \to 0$ to
\begin{equation*}
 C_d \phi(x/t) \Im \left( \frac{t (i\sgn t)^{d+1}}{(t^2 - |x - y|^2)^{\frac{d+1}{2}}} \right) \phi(y/t) = C_d t^{-d}   \phi(x/t) \Im \left( \frac{ i^{d+1}}{(1 - \frac{|x - y|^2}{t^2})^{\frac{d+1}{2}}} \right) \phi(y/t) .
\end{equation*}
When the dimension $d$ is even, this kernel is clearly of the form $\phi(x/t) \phi(y/t) t^{-d} W(|x - y|/t)$, where $W(z)$ is analytic in $|z| < 1$.

\end{document}